\documentclass[11pt,reqno]{amsart} 
\usepackage[latin1]{inputenc}
\usepackage[T1]{fontenc}
\usepackage{float}
\usepackage{pdfsync}
\usepackage{multirow}
\usepackage{color}
\usepackage{amssymb,comment,graphicx}
\usepackage{amsmath,amssymb,amsthm,amsfonts,amstext,amsbsy,amscd}
\usepackage{mathrsfs}
\usepackage{mathtools}
\usepackage{array}
\usepackage{rotating}
\usepackage{enumitem}
\usepackage{diagbox}
\usepackage{stackrel}
\usepackage{mathabx}
\usepackage[toc,page]{appendix}
 \usepackage[numbers]{natbib}

\newcommand{\Op}{\mathcal{O}_{\mathbb{P}}}

\newcommand{\var}{\mathrm{Var}}

\renewcommand{\tilde}{\widetilde}

\newcommand{\B}{B}
\theoremstyle{definition}
\newtheorem*{preuve}{Proof}
\newtheorem{lemma}{Lemma} 
\newtheorem{proposition}{Proposition} 
\newtheorem{Definition}{Definition} 
\newtheorem{remark}{Remark}

\newtheorem{ex}{Example}

\theoremstyle{theorem}
\newtheorem{theorem}{Theorem}[section]

\newtheorem{assumption}{Assumption}[section]

\renewcommand{\tilde}{\widetilde}
\usepackage{lscape}
\usepackage{ifpdf}
\ifpdf
\else
\usepackage{graphicx}
\DeclareGraphicsExtensions{.eps,.jpg,.png}
\fi
\usepackage{epstopdf}

\usepackage[]{algorithmic}
\usepackage[0]{algorithm}

\floatname{algorithm}{}
\definecolor{gris}{gray}{0.5}

\usepackage[bookmarks = true, colorlinks=true, linkcolor = blue, citecolor = blue, menucolor = black, urlcolor = black, plainpages=false]{hyperref}

\usepackage[left=2cm, top=3cm, right=2cm, foot=0.5cm]{geometry}

\author{Safa Ladgham $^{1,2}$, Rapha\"el Lachieze-Rey$^1$ }\thanks{$^1$ Universit\'e Paris Cité,  Laboratoire MAP5, UMR CNRS
	8145, 45 Rue des Saints-P\`eres, 75006 Paris\\indent $^{2}$ FSMP, safa.ladgham@parisdescartes.fr}
\begin{document}

\title { Local repulsion of planar Gaussian critical points
} 


\begin{abstract} 
	We study the local repulsion between critical points of a stationary isotropic smooth planar Gaussian field. We show that the critical points can experience a soft repulsion which is maximal in the case of the random planar wave model, or a soft attraction of arbitrary high order.  If the type of  critical points is specified (extremum, saddle point), the points experience a hard local repulsion, that we quantify with the precise magnitude of the second factorial moment of the number of points in a small ball.\\

	\smallskip
	
	{\bf Key words:}  Gaussian random fields; Stationary random fields; Critical points; Kac-Rice formula; repulsive point process.
	
	\smallskip
	
	{\bf AMS Classification:} 60G60- 60G15
	
\end{abstract}

 \maketitle

\section{Introduction}

The main topic of this paper is a local analysis of the critical points of a smooth stationary planar Gaussian field.
The study of critical points, their number as well as their positions, are important issues in various application areas such as sea waves modeling \cite{chevalier2013fast} , astronomy [\citealp*{larson2004hot},\citealp*{adler2007applications},\citealp*{lindgren1972local}] or neuroimaging [\citealp*{nichols2003controlling}, \citealp*{taylor2007detecting},\citealp*{worsley1996searching},\citealp*{worsley2004unified}]. In these situations, practitioners are particularly interested in the detection of peaks of the random field under study or in high level asymptotics of maximal points [\citealp*{cheng2017multiple},\citealp{taylor2007detecting},\citealp*{worsley1996searching}]. At the opposite of these Extremes Theory results, some situations require the topological study of excursion sets over moderate levels [\citealp*{adler2009random},\citealp*{cheng2016mean}] or the location study of critical points (not only extremal ones) [\citealp*{muirhead2020second}]. 

Repulsive point processes have known a surge of interest in the recent years, 
they are useful in a number of applications, such as sampling for quasi Monte-Carlo methods \cite{BarHar}, data mining, texture synthesis in Image Analysis \cite{LGA}, training set selection in machine learning, or numerical integration, see for instance  \cite{KulTas}, or as coresets for subsampling large datasets \cite{TBA}. 
 Critical points of Gaussian fields could be an alternative to determinantal point processes, which are commonly used for their repulsion properties despite the difficult issue of their synthesis  [\citealp*{desolneux:hal-01548767}].
Several definitions exist to characterize the repulsion properties of a stationary point process. We will use the following informal definition of local repulsion: A stationary random set of points $\mathcal  X\subset \mathbb{R}^{2}$ is \emph{locally repulsive at the second order} if, denoting by $\mathcal  N_{\rho }$ its number of points in a   ball centred in $0$ with radius $\rho $, we have
\begin{align}
\label{eq:ratio}
 \mathsf R_{\mathcal  N}:=\lim_{\rho \to 0} \frac{\mathbb{E}(\mathcal  N_{\rho }^{(2)})}{\mathbb{E}(\mathcal  N_{\rho })^{2}}<1
\end{align}
where for an integer $n, n^{(2)}=n(n-1)$ is the second order factorial power. This definition is motivated by the heuristic computation where we consider $x_{1}\neq x_{2}$  randomly sampled in $\mathcal  X\cap B_{1}$ and
\begin{align*}
\mathbb{E}(\mathcal  N_{\rho })&=\mathbb{P}(x_{1}\in B_{\rho })+\text{remainder}\\
\mathbb{E}(\mathcal  N_{\rho }^{(2)})&=\mathbb{P}(x_{2}\in B_{\rho }\;,x_{1}\in B_{\rho })+\text{remainder},
\end{align*}
where the remainder terms are hopefully negligible when $\rho $ is small.
In other words, a point process is locally repulsive if   the probability to find a point in a small ball diminishes if we know that there is already a point in this ball. The constant $\mathsf R_{\mathcal  N}$ is called the \emph{(second order) local repulsion factor}, it is a dimensionless parameter that is invariant under rescaling or rotation of the process $\mathcal  X$. It equals   $1$ if $\mathcal  X$ is a homogenous Poisson process, which is universally considered non-interacting. We say that the point process is \emph{weakly locally repulsive} (resp. \emph{attractive}) if   $\mathsf R_{\mathcal  N}\in (0,1)$ (resp. $(1,\infty )$), and strongly repulsive if  $\mathsf R_{\mathcal  N}=0$.

We study the repulsion properties of the stationary process $\mathcal  X_{c}$ formed by critical points of  a planar stationary isotropic Gaussian field $\psi $. We show that, depending on the covariance function of the field, they form a weakly locally repulsive or a weakly locally attractive point process, and that the minimal repulsion factor is $\mathsf  R_{\mathcal  X_{c}}=\frac{ 1}{8\sqrt{3}}$, reached when  $\psi $ is a Gaussian random wave model, which hence yields the most locally repulsive process of Gaussian critical points. There is on the other hand no maximal value for the limit. We also show that the subprocess formed by the local maxima of the field is strongly repulsive, as well as the subprocess formed by the saddle points, and give the precise magnitude   of the ratio decay  in the left hand member of \eqref{eq:ratio}.

Let us quote two recent articles that are concerned with a very similar question. The first one, which  has been a source of inspiration, is \cite{beliaev2019two}. In this paper, Belyaev, Cammarota and Wigman study the  repulsion of the critical points for a particular Gaussian field, the \emph{Berry's Planar Random Wave Model}, whose spectral measure is uniformly spread on a circle centred in $0$. They obtain the exact repulsion ratio for critical points and upper bounds for the repulsivity for specific types of critical points (saddle, extrema). Azais and Delmas \cite{azais2019mean} have studied the attraction or repulsion of critical points for general stationary Gaussian fields in any dimension. Using a different computation method, they get an upper bound for the second factorial moment which is compatible with the order of magnitude that we obtain. Their method is borrowed from techniques in random matrix theory, as suggested by Fyodorov \cite{fyodorov2004complexity}. Namely, an explicit expression for the joint density of GOE eigenvalues is exploited. 

In order to quantify the repulsion of the critical points, we compute the second factorial moment using the Rice or Kac-Rice formulas (see  \cite{adler2009random} or \cite{azais2009level} for details), as the vast majority of works concerned with counting the zeros or critical points of a random field. We get the asymptotics as the ball radius tends to 0 by performing a fine asymptotic analysis on the conditional expectations that are involved in the Kac-Rice formulas. 
\\~

The paper is organized as follows:   In Section \ref{sec:ass}, we present the Gaussian fields, which are the probabilistic object of our study, and the basic tools we will use for their study. In Section \ref{KAC RICE}, we derive the Kac-Rice formula, in a context  adapted to our framework.  The purpose of section \ref{first order} is to compute the expectation of the number of critical points and also the number of extrema, minima, maxima and saddle (see Proposition \ref{esperance of  critical cas general}). In Section \ref{second order}, we study the second factorial moment and discuss the repulsion properties of the  critical points.

\section{ Assumptions and  tools}
\label{sec:ass}

The main actors of this article are centered random Gaussian functions $\psi :\mathbb{R}^{2}\to \mathbb{R}$ whose law is invariant under  translations, and whose realisations are smooth. Formally it means that for $x_{1},\dots ,x_{n}\in \mathbb{R}^{2}$, $(\psi (x_{1}),\dots ,\psi (x_{n}))$ is a centered Gaussian vector which law is invariant under translation  of the $x_{i}$'s (and rotations if isotropy is further assumed), and that the sample paths $\{\psi (x);x\in \mathbb{R}^{2}\}$ are a.s. of class $\mathcal  C^{2}$ (or more).  See \cite{adler2009random}  for a rigourous and detailed exposition of Gaussian fields. Such a field is characterised by its reduced covariance function $\Gamma $ $$ \mathbb{E}[  \psi (z) \psi (w)  ]:=\Gamma (z-w)$$ for some $\Gamma :\mathbb{R}^{2}\to \mathbb{R}$, and if the field is furthermore assumed to be isotropic  (i.e. its law is invariant under rotations)   
\begin{align}
\label{eq:sigma}
\Gamma (z-w)=\sigma ( | z-w |^{2} )
\end{align} for some $\sigma :\mathbb{R}_{+}\to \mathbb{R}$, where $ | x | $ denotes the Euclidean norm of $x\in \mathbb{R}^{2}$.

We denote by $\nabla\psi (z)$  the gradient of $ \psi $ at $z\in \mathbb{R}^{2}$,  by  $ H_{\psi}(z)$ the Hessian matrix evaluated at $z$, when these quantities are well defined. For a smooth random field $\psi $, the set of critical points is denoted by 
\begin{align*}
\mathcal  X_{c}=\mathcal  X_{c}(\psi ):=\{x\in \mathbb{R}^{2}:\nabla \psi (c)=0\},
\end{align*} and the number of critical points in a  small disc $\B_{\rho }$ of radius $\rho>0$  is defined by
	$$\mathcal{N}^{c}_{\rho}(\psi ):=\#\mathcal  X_{c}\cap B_{\rho }.$$
	When there is no ambiguity about the random field $\psi $,  we simply write $\mathcal{N}_{\rho}^{c}$ instead of $\mathcal{N}_{\rho}^{c}(\psi )$.
	Similarly, we denote by resp. $\mathcal{N}_{\rho}^ {s}(\psi ), \mathcal{N}_{\rho}^{e}(\psi ), \mathcal{N}_{\rho}^{max}(\psi ), \mathcal{N}_{\rho}^{min}(\psi )$ the number of resp. saddles, extrema, maxima and minima, critical   points    {characterised by} the signs of the Hessian eigenvalues.	
 
As will be explained at Section \ref{second order}, to perform a second order local analysis of the repulsion of $\psi  $'s critical points, we must assume fourth order differentiation of $\psi  $, and for technical reasons we further assume that the fourth order derivative is $\alpha $-H\"older for some $\alpha >0$, we call this property $\mathcal{C}^{4+\alpha }$ regularity.
It is implied by $\sigma $ being of class $\mathcal{C}^{8+\beta }$ for some $\beta >2\alpha $, see Proposition \ref{prop:C4alpha} below. In this case, the H\"older constant is a random variable with Gaussian tail (see below).

\begin{assumption}
\label{ass:1}
Assume that $\psi  $ is a non-constant stationary  Gaussian field on $\mathbb{R}^{2}$ and its reduced covariance  $\Gamma  $ is of class $\mathcal{C}^{4+\beta }$ for some $\beta >0.$
\end{assumption}
This assumption implies the $\mathcal{C}^{4+\alpha }$ regularity of $\psi  $ by applying the proposition below to $\psi  $'s 4th order derivatives.
\begin{proposition}\label{prop:C4alpha}
Let $\varphi $ be a stationary Gaussian field $\mathbb{R}^{2}\to \mathbb{R}$, with reduced covariance function $\gamma:\mathbb{R}^{2}\to \mathbb{R}  $. Then if for some $C,\beta >0$,  for $\delta >0$ sufficiently small
\begin{align*}
 |\gamma  (x)- \gamma  (0) | \leqslant C | x | ^{\beta }, | x | \leqslant \delta ,
\end{align*}
then for $0<\varepsilon <\beta /2$ there is  a random variable  $U_{\varepsilon }$ with Gaussian tail such that for all $ x,y\in B _{1} $, 
\begin{align*}
 | \varphi  (x) - \varphi  (y) |\leqslant U _{\varepsilon } | x-y | ^{\beta /2-\varepsilon  }.
\end{align*}
\end{proposition}
\begin{proof}It follows from the classical result from Landau and Shepp \cite[(2.1.4)]{adler2009random} that  for a centred Gaussian field $f$ a.s. bounded on a Euclidean compact $T$, there is $c>0$   such that for large enough $u$, 
\begin{align*}
\mathbb{P}(\sup_{t\in T} | f(t) | \geqslant u)\leqslant 2\exp(-cu^{2}).
\end{align*}
We wish to apply this result to $T=\B_{1}\times \B_{1}$ and
\begin{align*}
f(x,y)= | x-y | ^{-\alpha }(\varphi   (x)-\varphi  (y)), (x,y)\in T.
\end{align*}
Let $\alpha =\beta /2-\varepsilon.$
The fact that $f$ is bounded is the consequence of the fact that $\varphi  $'s path are locally $\alpha  $-Holder 
for $\alpha <\beta /2$, see for instance \cite[Corollary 4.8]{Potthof}.\end{proof}
 
 \begin{Definition}
 \label{def:OP}Say that some random variables $X,Y$ satisfy $X=O_{\mathbb{P}(Y)}$ if $X\leqslant UY$ where $U$ is a random variable with a Gaussian tail, i.e.
\begin{align*}
\mathbb{P}( | U | >t)\leqslant c\exp(-c't^{2}),t\geqslant 0
\end{align*}for some $c<\infty ,c'>0.$
\end{Definition}
Proposition \ref{prop:C4alpha} hence implies that if a stationary field $\psi $'s reduced covariance $\Gamma $ is of class $\mathcal  C^{k+\eta }$, then 
\begin{align*}
\partial ^{k}\psi (t+h)=\partial ^{k}\psi (t)+\Op(h^{\eta /2}),t\in \mathbb{R}^{d}.
\end{align*}

\subsection{Dependency structure}

Stationarity conveys strong constraints on the dependence structure between the field's partial  derivatives at a given point. Let us recall formula \cite[(5.5.4)-(5.5.5)]{adler2009random}: if $\Gamma  $ is $\mathcal{C}^{k+\eta  }$ differentiable for some $k\in \mathbb{N},\eta  >0$, for natural integers $\alpha ,\beta ,\gamma ,\delta $ such that $\alpha +\beta \leqslant k,\gamma +\delta \leqslant k$, 

\begin{align*}
\mathbb{E}\left(
\partial _{1}^{\alpha }
\partial _{2}^{\beta }\psi (t)\cdot \partial _{1}^{\gamma }\partial _{2}^{\delta }\psi (s)\right)=
 \frac{ \partial^{\alpha +\beta +\gamma +\delta } }{\partial {t_{1}}^{\alpha } \partial {t_{2}}^{\beta  } \partial {s_{1}}^{\gamma  } \partial {s_{2}}^{\delta  }}\Gamma (t-s),s,t\in \mathbb{R}^{2}.
\end{align*}
In particular if $s=t$ we have the spectral representation
\begin{align}
\label{eq:spectral-deriv}
\mathbb{E}\left(
\partial _{1}^{\alpha }
\partial _{2}^{\beta }\psi (t)\cdot \partial _{1}^{\gamma }\partial _{2}^{\delta }\psi (t)\right)= (-1)^{\gamma +\delta  } \frac{\partial^{\alpha +\beta+ \gamma +\delta  } \Gamma(0)}{ {\partial { {{t}_{1}^{\alpha}}} }{\partial {t_{2}^{\beta}}  } {\partial {t_{1}^{\gamma}}}\;{\partial {t_{2}^{\delta}}} } =m_{\alpha +\gamma  ,\beta  +\delta }\text{\rm{ where }}m_{a,b}:=(-1)^{a }\imath^{a+b}\int_{\mathbb{R}^{2}}\lambda _{1}^{a  }\lambda _{2}^{b  }F (d\lambda ),t\in \mathbb{R}^{2}
\end{align}
where the symmetric  spectral measure $F $ is uniquely defined by 
\begin{align}
\label{eq:def-spectral}
\Gamma (t)=\int_{\mathbb{R}^{2}}\exp(-\imath   \lambda \cdot t)F (d\lambda ),t\in \mathbb{R}^{2}. 
\end{align}
Let us state important consequences of \eqref{eq:spectral-deriv}, and in particular of the fact that, due to the symmetry of $F $, the integral vanishes if $a$ or $b$ is an odd number. For this reason, $(-1)^{a}=(-1)^{b}$ when the integral does not vanish, and $m_{a,b}$ is symmetric in $a$ and $b.$\\

\begin{remark}
\label{rk:indep-deriv}
 For all $t\in \mathbb{R}^{2}, \psi (t)$ and $\partial _{j}\psi (t)$ are independent for $j=1,2$,  hence $\partial _{1}\psi  $ and $\partial _{2}\psi  $ are independent, and furthermore for any two natural integers $k,l$ which difference is odd, any partial derivatives of orders $k$ and $l$
\begin{align*}
\partial _{i_{1},\dots ,i_{k}}\psi(0)\text{\rm{ and }}\partial _{j_{1},\dots ,j_{l}}\psi(0)\text{\rm{ are independent.}}
\end{align*}
\end{remark}

Non-independence and technical difficulties will mainly emerge from dependence between  even degrees of differentiation of the field, such as $\psi (t) $ and $\partial _{11}\psi  (t)$, or $\partial _{11}\psi(t)  $ and $\partial _{22}\psi (t) $, or between the values of the field at different locations, say $\psi (s)$ and $\psi (t),s\neq t.$
 {A  case we must discard is that of constant $\psi  $, i.e. $\psi (t) =U$ for some Gaussian variable $U$, and this is what we call a {\it trivial} Gaussian field.}
 
 Also, Cauchy-Schwarz inequality yields that for $\alpha ,\beta ,\gamma ,\delta \in \mathbb{N}$
\begin{align*}
 | m_{\alpha +\gamma ,\beta +\delta } | ^{2}\leqslant   m_{2\alpha ,2\beta }m_{2\gamma ,2\delta } ,
\end{align*}
and there is equality only if $\lambda _{1}^{\alpha }\lambda _{2}^{\beta }$ is proportionnal to $\lambda _{1}^{\gamma }\lambda _{2}^{\delta }$ $dF $-a.s. In the isotropic case (i.e. $F$ is invariant under spatial rotations), unless $\alpha =\gamma ,\beta =\delta $ it can only happen if $dF $ is the Dirac mass in $0$, i.e.
\begin{align}
\label{eq:CS-lambda}
m_{\alpha +\gamma ,\beta +\delta }^{2}< m_{2\alpha ,2\beta }m_{2\gamma ,2\delta },\alpha \neq \gamma \text{\rm{  or }}\beta \neq \delta  \text{\rm{ if }}\psi   \text{\rm{ is non-trivial isotropic}}.
\end{align}
  
	\begin{proposition}
		\label{prop covari et variance en deriv}
		Let $\psi $ be an isotropic   Gaussian field $\mathbb{R}^2 \rightarrow \mathbb{R}$ that satisfies Assumption \ref{ass:1} with covariance under the form \eqref{eq:sigma}. We   indicate   the  first   derivatives  of $\sigma  $  at point  ${0} \in \mathbb{R}$  by  $\sigma ^{\prime  }( 0)=\eta_{0},$ $ \sigma ^{\prime \prime }  (0)=\mu_{0},$ $\sigma ^{(3)} (0)=\nu_{0}, $ $\sigma^{{4}} (0)=\upsilon$. Then
		\begin{align}
		\label{eq:sigma-deriv}
		\var(\partial _{i}\psi(0))&=-2 \eta_{0}=m_{2,0}>0,\\
		\notag\var(\partial _{12}\psi(0))&=2^2 \mu_{0}=m_{2,2}>0,\\\notag
		\var(\partial _{ii}\psi(0))&=3\cdot 2^2 \mu_{0}=m_{4,0}>0, i=1,2,\\\notag
		\var(\partial _{iii}\psi(0))&=-15 \cdot 2^3 \nu_{0}=m_{6,0}>0, i=1,2.
		\end{align}
	\end{proposition}

The two last equalities illustrate the fact that isotropy and polar change of coordinates yield other relations between the $m_{a,b}$ of the form 
\begin{align*}
m_{a,b}=\alpha _{a,b}m_{a+b,0}
\end{align*}
where the coefficients $\alpha _{a,b}$ don't depend on $F $.

\begin{ex} 
\label{ex:grw} Let $J_{0}$ be the Bessel function of the first order
\begin{align*}
J_{0}(x)=\frac{ 1}{2\pi }\int_{ 0}^{2\pi }e^{-ix\cos(\theta )}d\theta   ,x\in \mathbb{R}.
\end{align*}
For $k >0$ let $\psi $ be the \emph{Gaussian random wave} with parameter $k$, i.e. the isotropic stationary Gaussian field with reduced covariance function 
\begin{align*}
\Gamma (z)=J_{0}( k  | z | ).
\end{align*}
As is apparent from \eqref{eq:def-spectral}, this is the  centered Gaussian field whose spectral measure is the uniform law on the centred circle with radius $k$. It is important as it is the unique (in law) stationary Gaussian  field for which 
\begin{align*}
\partial _{11}\psi +\partial _{22}\psi +k^{2} \psi =0\text{\rm{  a.s. }}
\end{align*}
up to a multiplicative constant. See for instance \cite{beliaev2019two,MRV,NPR} and references therein for recent works about diverse aspects of planar random wave models.  As proved at Section \ref{sec:dependency}, it is the only non-trivial stationary  isotropic stationary field satisfying a  linear  partial differential equation of order three or less.
As critical points are not modified by adding a constant, we also consider {\it shifted Gaussian random waves} (SGRW), of the form $\tau  U+\sigma \psi $, where $\tau   \geqslant 0,\sigma >0,\psi $ is a GRW and $U$ is an independent centered standard Gaussian variable.  The spectral measure of a SGRW is the sum of a uniform measure on a circle of $\mathbb{R}^{2}$ centred in $0$ and a finite mass in $\{0\}.$
\end{ex}

\section{The Kac-Rice formula} 
\label{KAC RICE} 
The Kac-Rice formula gives a description of the   factorial moments of the zeros of a random field. Let us give a formula adapted to counting the critical points of a certain type. The following result can be proved by combining the proofs of Theorems 6.3 and 6.4 from \cite{azais2009level}, see also \cite[Appendix A]{azais2019mean}.

\begin{theorem}
\label{thm:KR}Let $\psi $ isotropic satisfying Assumption \ref{ass:1}. Let $k\in \{1,2\}$, $B_{1},B_{2}$  some open subsets of $\mathbb{R}^{d}$, 
\begin{align*}
\mathcal  N_{\rho }^{B_{i}}=\{t\in B(0,\rho ):\nabla \psi (t)=0, H_{\psi }(t) \in B_{i}\}.
\end{align*} Then  {for $\rho $ sufficiently small}
\begin{align}
	\notag
	\mathbb{E}[ \mathcal{N}_{\rho }^{B_{1}} ]=   \int_{\mathcal  B_{\rho }}  \mathrm{K}^{B_{1}}_{1}( {t}) \, \mathrm{d} {t}, \\
	\label{eq:KR2}
	\mathbb{E}[ \mathcal{N}_{\rho }^{B_{1}}   (\mathcal{N}_{\rho }^{B_{2}}-1) ]=   \int_{\mathcal  B_{\rho }^{2}}  \mathrm{K}^{B_{1},B_{2}}_{2}(t_{1},t_{2}) \, \mathrm{d}{t}, 
	\end{align}
	where   we have  the $k$-point correlation function :
	 
\begin{align*}
	\mathrm{K}_{1}^{B_{1}}(t)=&\phi_{\nabla \psi (t)}({0}) \; \mathbb{E}\left[ |  \det H_{\psi }({t} ) |  \;\;   \mathbf{1}_{B_{1}}(H_{\psi }({t} )) \; \Big| \;
	\nabla \psi (t)=0 \right],\\
	\mathrm{K}_{2}^{B_{1},B_{2}}(t_{1},t_{2})=&\phi_{(\nabla \psi (t_{1}),\nabla \psi (t_{2}))}({0},{0}) \; \mathbb{E}\left[\prod_{i=1}^2 | \det H_{\psi }({t}_{i})| \;\;   \mathbf{1}_{B_{i}}(H_{\psi }({t}_{i})) \; \Big| \;
	\nabla \psi (t_{1})=\nabla \psi (t_{2})=0 \right],
\end{align*}
	where $\phi_{V}$ is the density probability function of  a Gaussian vector $V$.
	\end{theorem}
	
	 We are specifically interested in a finite class of sets $B_{i}$, namely 
\begin{align*}
B_{c}=&\mathcal  M_{d}(\mathbb{R})\text{ the class of $d\times d$ square matrices},\\
B_{ext}=&\det^{-1}((0,\infty ))\\
B_{s}=&\det^{-1}((-\infty ,0)),\\
B_{min}=&\{H\text{ definite positive}\},\\
B_{max}=&\{H\text{ definite negative}\}.
\end{align*}
In this case, the exponent in $\mathcal  N$ or $K_{i}$ is replaced by the subscript of $B$, e.g. 
\begin{align*}
\mathcal  N_{\rho }^{c}=\mathcal  N_{\rho }^{B_{c}},K_{2}^{s,s}=K_{2}^{B_{s},B_{s}}, etc...
\end{align*}

\begin{proof}
With $Z(t)=\nabla \psi (t),Y^{t}=H_{\psi }(t),g(t,H)=\mathbf{1}_{\{\det(H)\in B\}}$, we have
\begin{align*}
\mathcal  N_{\rho }^{B}=\sum_{t:Z(t)=0}g(t,Y^{t}).
\end{align*}
Let us show that hypothesis (iii') of \cite[Th.6.3]{azais2009level} is satisfied, that is for $\rho $ small enough and $t,s\in \mathcal  B_{\rho }$, the law of $(\nabla \psi (t),\nabla \psi (s))$ is non-degenerated.
Let us expand 
\begin{align*}
\mathbb{E}(\partial _{i}\psi (s)\partial _{j}\psi (t))=\frac{ \partial^{2} }{\partial _{s_{i}}\partial _{t_{j}}}\sigma ( | s-t | ^{2})= \begin{cases}
-2\sigma '( | s-t | ^{2})-4(s_{i}-t_{i})^{2}\sigma ''( | s-t | ^{2})$ if $i=j\\
-4(s_{i}-t_{i})(s_{j}-t_{j})\sigma ''( | s-t | ^{2})$ if $i\neq j.
\end{cases}
\end{align*}

By isotropy it suffices to evaluate it in $t=(r,0), s=(-r,0)$ for $r\geqslant 0$. Let us write the $4\times 4$ covariance matrix in function of $\eta _{r}=\sigma '(4 r^{2}),\mu_{r}=\sigma ''(4 r^{2})$
\begin{align}
\label{eq:Sigma}
\Sigma =-2\left(\begin{array}{cc}\eta_{0}I & \eta_{r}I+2\mu _{r}A_{r}\\ \eta _{r}I+2\mu _{r}A_{r}& \eta_{0}I\end{array}\right)\end{align}
where 
\begin{align*}
A_{r}=\left(\begin{array}{cc} 4 r^2 & 0 \\0 &0\end{array}\right).
\end{align*}

Hence the block determinant is 
\begin{align*}
16\det(\eta _{0}^{2}I-(\eta _{r}I+2\mu _{r}A_{r})^{2})&=16\det((\eta	_{0}^{2}-\eta _{r}^{2})I-4\mu _{r}\eta_{r}A_{r}-4\mu _{r}^{2}A_{r} ^{2})\\
&=16(\eta _{0}^{2}-\eta _{r}^{2})((\eta _{0}^{2}-\eta _{r}^{2})-16\mu _{r}\eta _{r}r^{2}-64\mu _{r}^{2}r^{4}).
\end{align*}
 
This is equivalent to 
\begin{align*}
16\cdot 8\eta_{0}(-\mu _{0}r^{2})(8\eta	 _{0}(-\mu _{0}r^{2})+\Op(r^{2})-16\mu _{0}\eta_{0}r^{2}+\Op(r^{2}))\sim-128\eta _{0}\mu _{0}r^{2}(-24\eta _{0}\mu _{0}r^{2})=3 \cdot 2^{10}\mu _{0}^{2}\eta _{0}^{2}r^{4},
\end{align*}
where we have  $\mu _{0}\eta _{0}\neq 0$ in virtue of \eqref{eq:sigma-deriv}. Hence the determinant is non zero for $r\neq 0$ sufficiently small.
Then the modification of the proof of Theorem 6.3 following the proof of Theorem 6.4 of \cite{azais2009level} yields the result, see Appendix A in \cite{azais2019mean}.

It yields in particular
\begin{align}
\label{eq:det-cov-grad}
\phi_{(\nabla \psi (t_{1}),\nabla \psi (t_{2}))}({0},{0})=\frac{ 1}{\sqrt{\det(\Sigma )}(2\pi)^{2} }=\frac{ 1}{2^{7}\pi^{2} \sqrt{3}  | \mu _{0}\eta _{0} | r^{2}}(1+o_{r\to 0}(1)).
\end{align}
\end{proof}


\section{First order}
\label{first order}
In this section, we are interested in  the computation of the expectated number of critical points  in a Borel set $B\subset \mathbb{R}^{2}.$ \\
\begin{proposition}
	\label{esperance of  critical cas general}
	Let $\psi=\{ \psi (z): z\in \mathbb{R}^2\}$ be a   {non-trivial isotropic stationary} Gaussian field  $\mathbb{R}^2 \rightarrow \mathbb{R}$ which is a.s. of class  $\mathcal{C}^{2}$ and let $\sigma $ be defined by \eqref{eq:sigma}.  Let\begin{align*}
\lambda _{c}=\frac{4 }{\sqrt{3}\pi } \frac{\sigma ''(0) }{(-\sigma '(0))  }.
\end{align*}
In virtue of Proposition \ref{prop covari et variance en deriv}, $\lambda _{c}\in (0,\infty ).$
 Then, for every   $\rho>0$,  we have 
	\begin{align}
	\label{expection of critcial point cas generale hh}
	\mathbb{E}[\mathcal{N}_{\rho}^{c}]= \lambda _{c} | B_{\rho } | ,\\
	\label{expection of s}
	\mathbb{E}[ \mathcal{N}_{\rho}^{e}]=
	\mathbb{E}[ \mathcal{N}_{\rho}^{s}]=\frac{ 1}{2}\mathbb{E}[\mathcal{N}_{\rho}^{c}],\\
	\notag\mathbb{E}[\mathcal{N}_{\rho}^{min}]=  \frac{ 1}{4}\mathbb{E}[\mathcal{N}_{\rho}^{c}] .
	\end{align}
\end{proposition}

A sufficient condition for $\psi  $ being of class $\mathcal{C}^{2}$ is that $\sigma $ is of classe $\mathcal{C}^{4+\beta }$ for some $\beta >0$, see Proposition \ref{prop:C4alpha}.

\begin{remark}
By stationarity, $\lambda _{c}$ is the intensity of $\mathcal  X_{c}(\psi )$, i.e. the mean number of critical points per unit volume.\end{remark}

\begin{preuve}According to Theorem \ref{thm:KR}, we must simply evaluate	\begin{equation*}
	\mathrm{K}_{1}(z)= \phi_{\nabla \psi (z)} (0,0) \;  \; \mathbb{E} \left[ \; \mid \det H_{\psi}(z) \mid \big| \nabla\psi (z)=0 \; \right] .
	\end{equation*}
	The stationarity   of $\psi $ implies that  $K_{1}(z)$ is independent of $z$, see formula \eqref{eq:KR2}. So, we get 
	\begin{equation}
	\label{esperace du point critique ca generale cas generale  }
	\mathbb{E}[ \mathcal{N}_{\rho}^{c} ]= | B_{\rho }  | \mathrm{K}_{1} (0) \mbox{\qedhere}.
	\end{equation}
	
	Using the matrix   $\Sigma $ with $r=0$ in \eqref{eq:Sigma}, we immediately obtain the probability density function of (two-dimensional vector) $\nabla \psi (z)$ evaluated at  point $(0,0)$: 
	\begin{equation}
	\label{psi en 0 cas genrela e}
	\phi_{\nabla\psi (z)}(0,0)= \frac{1}{2 \pi}  \frac{1}{ \sqrt{4 \eta_{0}^2}} =\frac{1}{4 \pi  |\eta_{0}|},
	\end{equation}
	where $\eta _{0}=\sigma '(0).$
	From this point  until the end of the proof we will use the method of the article  \cite{beliaev2019two}. Since the first and the second derivatives of $\psi(z)$ are independent at every fixed point $z \in \mathbb{R}^{2},$ then:
	\begin{equation}
	\label{resultat de esperzcne conditonale cas generale  }
	\mathbb{E}\left[ \left|\; \det H_{\psi}(z) \right|  \;  \big| \nabla \psi (z)=0 \right] = \mathbb{E}\left[  \; \left|   \det H_{\psi}(z)\right|  \; \right]=
	\mathbb{E} \left[ \;  \left| \det H_{\psi}(z)\right|    \; \right]=\mathbb{E}\left[ \;  \left| \partial_{11}\psi(0)  \partial_{22}\psi(0)- \partial^{2}_{12}\psi(0) \right|  \; \right].
	\end{equation}
	To evaluate \eqref{resultat de esperzcne conditonale cas generale  }, we consider the transformation $W_{1}= \partial_{11}\psi(0)$, $W_{2}= \partial_{12}\psi(0)$, $W_{3}= \partial_{11}\psi(0)+ \partial_{22}\psi(0)$ and we write   $\mathbb{E}\left[ \;  \left| \partial_{11}\psi(0)  \partial_{22}\psi(0)- \partial^{2}_{12} \psi(0)\right|  \; \right]$  in terms of a conditional expectation as follows: 
	\begin{align}
	\label{transformation de  vecteur Y cas generale }
	\mathbb{E}\left[ \;  \left| \partial_{11}\psi(0)  \partial_{22}\psi(0)- \partial^{2}_{12}\psi(0) \right|  \; \right]& = \mathbb{E}\left[ \;   \left| W_{1} W_{3}-W^{2}_{1} -W^{2}_{2} \right| \; \right]= \mathbb{E}\left[ \;  \mathbb{E}\left[  \; |W_{1} W_{3}-W^{2}_{1}-W^{2}_{2} |  \;  \big| W_{3}\;  \right]  \right],
	\end{align}
	where $W=(W_{1} ,W_{2}, W_{3}) $ is a centered Gaussian  vector field with covariance matrix $D.$\\
	Use Proposition \ref{prop covari et variance en deriv} and Remark \ref{rk:indep-deriv}, we have 
	$ D=\begin{pmatrix}
	12  \mu_{0} &0&16  \mu_{0} \\
	0& 4 \mu_{0} &0\\
	16 \mu_{0}  &0& 32 \mu_{0}
	\end{pmatrix}.
	$\\
	The conditional distribution of $(W_{1},W_{2})\big|W_{3}$  is Gaussian with covariance matrix  $\Sigma_{(W_{1},W_{2})|W_{3}}=\begin{pmatrix}
	4 \mu_{0} &0\\
	0& 4 \mu_{0}
	\end{pmatrix}$
	and expectation
	$ \mathbb{E}[ \; (W_{1},W_{2}) \; \big| W_{3}=t]=\begin{pmatrix}
	\frac{t}{2} \\\\
	0
	\end{pmatrix}$ for $t\in \mathbb{R}.$\\
	The conditioned Gaussian vector $(W_{1},W_{2})|W_{3}=t$   is distributed as  $(2 \sqrt{\mu_{0}} Z_{1}+	\frac{t}{2},  2 \sqrt{\mu_{0}}  \; Z_{2})$ where $Z_{1},$ $Z_{2}$ are  two independent standard Gaussian random variables, hence we have 

	\begin{align*}
	\mathbb{E}\left[ \;  \left| W_{1} W_{3}-W^{2}_{1}-W^{2}_{2} \right| \;  \big| \;  W_{3}=t \;  \right]&=\mathbb{E}[ \;  \mid W_{1} t-W^{2}_{1}-W^{2}_{2} \mid | W_{3}=t \;  ] ]\\
	&=\mathbb{E}[ \;  |( 2 \sqrt{\mu_{0}} Z_{1}+t/2)t-(2 \sqrt{\mu_{0}}  Z_{1}+t/2)^{2}- 4 \mu_{0} \; Z^{2}_{2} |\;]\\
	&=\mathbb{E}[ \; | - 4 \mu_{0} Z_{1}^{2} - 4 \mu_{0} Z_{2}^{2} + \frac{t^{2}}{4} |  \; ]\\
	&= 4 \mu_{0} \; \mathbb{E}\left[\;\left|-X+  \frac{t^{2}}{16 \mu_{0}}\right| \right],
	\end{align*}
	 where  $X$ is a  $\chi-$square random variable with  density 
	$f_{X}(x)=\frac{1}{2} e^{-\frac{x}{2}}, x>0.$\\
	So  
	\begin{align*} 
	\mathbb{E}[\;|-X+ \frac{t^{2}}{16 \mu_{0}}| \;]&=\frac{1}{2} \int_{0}^{ \frac{t^{2}}{16 \mu_{0}}}( \frac{t^{2}}{16 \mu_{0}}-x) e^{-\frac{x}{2} }dx +\frac{1}{2} \int_{ \frac{t^{2}}{16 \mu_{0}}}^{+\infty }(- \frac{t^{2}}{16 \mu_{0}}+x) e^{-\frac{x}{2}} dx\\
	&=-2 +4 e^{-\frac{t^{2}}{32 \mu_{0}}}+\frac{t^{2}}{16 \mu_{0}}, 
	\end{align*}
	then 
	\begin{align}
	\label{expectation de Y  cas generale }
	\mathbb{E}\left[ \;  \left| W_{1} W_{3}-W^{2}_{1}-W^{2}_{2} \right| \;  \big| \;  W_{3}=t \;  \right]&=  4 \mu_{0} \; \frac{1}{8 \sqrt{\pi\mu}} \int_{\mathbb{R}}  \mathbb{E}\left[\;\left|-X+  \frac{t^{2}}{16 \mu_{0}}\right| \right]\;  e^{-\frac{t^{2}}{64 \mu_{0}}} dt\nonumber\\
	&= \frac{\sqrt{\mu_{0}}}{2 \sqrt{\pi}} \int_{\mathbb{R}} e^{-\frac{t^{2}}{64 \mu}} \left(-2 +4 e^{-\frac{t^{2}}{32 \mu_{0}}}+\frac{t^{2}}{16 \mu_{0}} \right)  \;  dt=  \frac{16 \mu_{0}}{ \sqrt{3}}.
	\end{align} 
	By combining   Equations  \eqref{esperace du point critique ca generale cas generale  }, \eqref{psi en 0 cas genrela e},\eqref{resultat de esperzcne conditonale cas generale  }  and \eqref{expectation de Y  cas generale  }, we obtain Formula \eqref{expection of critcial point cas generale hh}.
	
	Now, we turn to the evaluation of the expected number of the  extrema and saddle points. We have
	
	\begin{align*}
	\mathcal{N}^{e}_{\rho}:=\mathcal  N_{\rho }^{(0,\infty )}&=\string # \{x \in \B_\rho:  \nabla \psi (x)=0, \det H_{\psi}(z)>0\}\\
	\mathcal{N}^{s}_{\rho}:=\mathcal  N_{\rho }^{(-\infty ,0 )}&=\string # \{x \in \B_\rho: \nabla \psi (x)=0, \det H_{\psi}(z)<0\}.
	\end{align*}
	As  previously,  we apply   the Kac-Rice formula   from  Section \ref{KAC RICE}.  We get: 
	
	$$ \mathbb{E}[ \mathcal{N}_{\rho}^{e} ]=    \int_{\B_\rho}  \mathrm{K}^{e}_{1}(z)   dz \qquad \mbox{and} \qquad  \mathbb{E}[ \mathcal{N}_{\rho}^{s} ]=    \int_{\B_\rho}  \mathrm{ K}^{s}_{1}(z)  dz,$$
	
	where $$ \mathrm{K}_{1}^{e}(z)= \phi_{\nabla \psi (z)} (0,0) \;  \; \mathbb{E} \left[ \; | \det H_{\psi}(z)|  \;  \mathbf{1}_{ \{ \det H_{\psi}(z)>0 \}} \;  \big| \nabla\psi (z)=0 \; \right], $$
	$$ \mathrm{K}_{1}^{s}(z)= \phi_{\nabla \psi (z)} (0,0) \;  \; \mathbb{E} \left[ \; |{\det H}_{\psi}(z)| \; \mathbf{1}_{  \{\det H_{\psi}(z)<0 \}}  \; \big| \nabla\psi (z)=0 \; \right]. $$
	Since  the first and the second  derivatives of $\psi(z)$ are independent at every fixed point $z \in \mathbb{R}^{2},$ we obtain 
	\begin{equation}
	\label{esperance de esxtema en genrla en preuve}
	\mathbb{E}[ \mathcal{N}_{\rho}^{e} ]= \pi \rho^{2}   \phi_{\nabla \psi (z)} (0,0) \;  \; \mathbb{E} \left[ \; | \det H_{\psi}(z)|  \;  \mathbf{1}_{ \{\det H_{\psi}(z)>0\} }  \right],
	\end{equation}
	\begin{align}
	\label{s espere en gelrz en preuve}
	\mathbb{E}[ \mathcal{N}_{\rho}^{s} ]&= \pi \rho^{2}  \phi_{\nabla \psi (z)} (0,0) \;  \; \mathbb{E} \left[ \; | \det H_{\psi}(z)|  \;  \mathbf{1}_{ \{  \det H_{\psi}(z)<0 \} }  \right].
	\end{align}

	Using the same argument as in the case  of critical points, we  write
	\begin{align} 
	\label{forme derne de extreme en fonction de det genrale }
	\mathbb{E} \left[ \;  \left| \det H_{\psi}(z)\right|    \mathbf{1}_{  \{det H_{\psi}(z)>0 \}} \; \right]&=  \mathbb{E}\left[  | \partial_{11}\psi(0) \partial_{22}\psi(0)-{\partial}^{2}_{12}\psi(0)  | \; \mathbf{1} _{\{\partial_{11}\psi(0) \partial_{22}\psi(0)-\partial^{2}_{12}>0\}} \right]\nonumber\\
	&= 4 \mu \; \frac{1}{8 \sqrt{\pi\mu}} \int_{\mathbb{R}}  \mathbb{E}\left[\;\left|-X+  \frac{t^{2}}{16 \mu_{0}}\right|   \mathbf{1} _{\{-X+  \frac{t^{2}}{16 \mu_{0}}>0\}}\right]\;  e^{-\frac{t^{2}}{64 \mu_{0}}} dt\nonumber\\
	&=\frac{\sqrt{\mu_{0}}}{2 \sqrt{\pi}}\int_{\mathbb{R}}  \left(-2 +2 e^{-\frac{t^{2}}{32 \mu_{0}}}+\frac{t^{2}}{16 \mu_{0}} \right)  \;   e^{-\frac{t^{2}}{64 \mu_{0}}} dt\nonumber\\
	&=  \frac{8 \mu_{0}}{ \sqrt{3}},
	\end{align}
	and 
	\begin{align} 
	\label{forme derne de saddle esperance en fonction de det  genrale }
	\mathbb{E} \left[ \;  \left| \det H_{\psi}(z)\right|    \mathbf{1}_{  \{det H_{\psi}(z)<0\}} \; \right]&= \mathbb{E}\left[ | \partial_{11}\psi(0) \partial_{22}\psi(0)-\partial^{2}_{12} |  \; \mathbf{1} _{\{\partial_{11}\psi(0) \partial_{22}\psi(0)-\partial^{2}_{12}<0\}}  \right]\nonumber\\
	&= 4 \mu \; \frac{1}{8 \sqrt{\pi\mu_{0}}} \int_{\mathbb{R}}  \mathbb{E}\left[\;\left|-X+  \frac{t^{2}}{16 \mu_{0}}\right|   \mathbf{1} _{\{-X+  \frac{t^{2}}{16 \mu_{0}}<0\}}\right]\;  e^{-\frac{t^{2}}{64 \mu_{0}}} dt\nonumber\\
	&=\frac{\sqrt{\mu_{0}}}{2 \sqrt{\pi}} \int_{\mathbb{R}}  \left(2 e^{-\frac{t^{2}}{32 \mu_{0}}}\right)  \; e^{-\frac{t^{2}}{64 \mu_{0}}} dt\nonumber\\
	&=\frac{{8 \mu_{0}}}{ \sqrt{3}}. 
	\end{align}
	
	By combining  Equations \eqref{psi en 0 cas genrela e}, \eqref{esperance de esxtema en genrla en preuve}, \eqref{s espere en gelrz en preuve}, \eqref{forme derne de extreme en fonction de det genrale } and \eqref{forme derne de saddle esperance en fonction de det  genrale }, we obtain  Formula \eqref{expection of s}.\\
	Finally, we turn  to the calculation of the expectation  of the number of minima and maxima  in $\B_\rho$.\\
	We know that:  \\
	$$\mathcal{N}_{\rho}^{e}=\mathcal{N}_{\rho}^{min}+\mathcal{N}_{\rho}^{max}$$  so  $\mathbb{E}[ \mathcal{N}_{\rho}^{e}]=\mathbb{E}[\mathcal{N}_{\rho}^{min}]+\mathbb{E}[\mathcal{N}_{\rho}^{max}].$
	
	By symmetry of the Gaussian field $\psi $, we have the following equality:  $\mathcal{N}_{\rho}^{max}(-\psi ) \overset{\mathcal{L}}{=}\mathcal{N}_{\rho}^{max}(\psi )=\mathcal{N}_{\rho}^{min}(-\psi )$ \quad  for \quad  $-\psi  \overset{\mathcal{L}}{=} \psi  ,$ \; 
	therefore  \; $ \mathbb{E}[\mathcal{N}_{\rho}^{min}]=\mathbb{E}[\mathcal{N}_{\rho}^{max}].$
	
	Finally, we obtain 
	$$ \mathbb{E}[\mathcal{N}_{\rho}^{min}]=\mathbb{E}[\mathcal{N}_{\rho}^{max}]=\frac{1}{2} \mathbb{E}[\mathcal{N}_{\rho}^{e}].$$
	
\end{preuve}

\section{Second order}
\label{second order}
In this section, we will study the asymptotic  behaviour of the second factorial  moment of  $\mathcal{N}^{c}_{\rho}$  when $\rho$ goes to zero. The following theorem is the main result of this paper. Given two quantitites $\alpha _{\rho },\beta _{\rho },$ write $\alpha _{\rho }\asymp \beta _{\rho }$ if for two constants $0<c<c'<\infty ,$ we have $ {c}\alpha _{\rho }\leqslant \beta _{\rho }\leqslant c'\alpha  _{\rho }$ for $\rho $ sufficiently small, and $\alpha _{\rho }\sim \beta _{\rho }$ if $\alpha _{\rho }/\beta _{\rho }\to 1$, with the convention $0/0=1.$
\begin{theorem}
	\label{theom second 2 crtique critque and extrema and s}
	Let $\psi $ be an isotropic   Gaussian field $\mathbb{R}^2 \rightarrow \mathbb{R}$ that satisfies Assumption \ref{ass:1}.The \emph{repulsion factor}  $\mathsf  R_{c}:=\mathsf  R_{\mathcal  X_{c}}$ is given by
\begin{align*}
\mathsf R_{c}=\frac{ \sqrt{3}}{8}\left(
5\frac{\sigma '''(0)\sigma '(0) }{(\sigma ''(0))^{2}}-3
\right) .
\end{align*} As  $ \rho \rightarrow 0,$  we have the following asymptotic equivalent expression   for the second factorial moment of the  number of critical points
	\begin{equation}
	\label{result of the second factorial moment critical points theroem}
	\mathbb{E}[ \mathcal{N}^{c}_{\rho} (\mathcal{N}^{c}_{\rho}-1)] \sim \mathsf R_{c}\lambda _{c}^{2}  { | B_{\rho } | ^{2}}  \asymp \rho ^{4}.
	\end{equation}
Depending on the law of $\psi $, $\mathsf R_{c}$ can take any prescribed value in $[\frac{ 1}{8\sqrt{3}},\infty )$, and $ \frac{ 1}{8\sqrt{3}}$ is the minimal possible value, it is reached  iff $\psi $ is a shifted Gaussian random wave (Example \ref{ex:grw}).\\

	For the numbers  of  extrema, saddles  in a ball of radius $\rho$,  we have as $\rho \to 0$
	\begin{equation}
	\label{result od the second moment of extrema points }
	  \mathbb{E}[ \mathcal{N}^{e}_{\rho} (\mathcal{N}^{e}_{\rho}-1)] \asymp  \rho^7 ,
	\end{equation}
	\begin{equation}
	\label{result od the second moment of saddles points }
	 \mathbb{E}[ \mathcal{N}^{s}_{\rho} (\mathcal{N}^{s}_{\rho}-1)] \asymp  \rho^7 \ln(\rho ),
	\end{equation}
	\begin{equation}
	\label{preuve crituqe crique vzaleur ac}
	\mathbb{E}[ \mathcal{N}_{\rho}^{e}\mathcal{N}_{\rho}^{s}]  \sim  \mathsf R_{c}\lambda _{c}^{2} | B_{\rho } | ^{2}.
	\end{equation}

\end{theorem} 

\begin{remark}The repulsion factor terminology comes from $\lambda _{c} | B_{\rho } | \sim \mathbb{E}(\mathcal  N_{\rho }) $ and by the heuristic explanation after \eqref{eq:ratio}.
\end{remark}

\begin{remark}
By truncating the expansion of the type \eqref{detr hess au point z} at a lower order, one could prove that Expression \eqref{result of the second factorial moment critical points theroem} is  valid under the weaker assumption that $\Gamma $ is of class $\mathcal{C}^{6+\beta }$ (and $\psi $ is of classe $\mathcal{C}^{3}$).
\end{remark}

\begin{ex}[Bargmann Fock field]
Consider the Bargmann-Fock field with parameter $k$, which is the stationary isotropic Gaussian field with reduced covariance function
\begin{align*}
\sigma (r)=\exp(-kr),r\geqslant 0.
\end{align*}
According to Proposition \ref{esperance of  critical cas general}, we have for the first order
	\begin{align*}
\sigma '(0) =& -k \\
\sigma ^{''}(0) =& k^{2}\\
\sigma ^{'''}(0)=&  -k^{3}\\
\mathbb{E}[\mathcal{N}_{\rho}^{c}]=& \frac{4 }{ \sqrt{3}} k \; \rho^{2}.\\
\end{align*}
Hence the attraction factor is 
\begin{align*}
\mathsf R_{c}=\frac{ \sqrt{3}}{4}<1,
\end{align*}
which means that the process of critical points is locally weakly repulsive.  It logically does not depend on the scaling factor $k.$
	  
\end{ex}

\begin{ex}
 [Gaussian random waves]
\label{ex:factorial-grw}
		Consider the Gaussian random wave introduced at Example \ref{ex:grw}
		by
	$\sigma(x)=J_{0}(k \sqrt{ x} ),x\geqslant 0$. We have : 
	
	\begin{align*} 
\sigma ' (0)=&  -\frac{k^2}{4}\\
\sigma ^{''} (0)= &   \frac{ k^{4}}{2^{5}}\\
\sigma ^{''} (0)=&  -\frac{ k^6 }{3\cdot 2^7}  \\
\lambda _{c}=&\frac{ k^{2}}{2\sqrt{3}\pi } \text{ i.e. }\mathbb{E}[\mathcal{N}_{\rho}^{c}]= \frac{k^2}{2 \sqrt{3}} \; \rho^{2},\\
	\end{align*}
	Hence the attraction factor takes the smallest possible value
\begin{align*}
\mathsf R_{c}=\frac{ 1}{8\sqrt{3}},
\end{align*}
which means that the process of critical points is locally weakly repulsive.  We retrieve the second factorial moment of Beliaev, Cammarota and Wigman \cite{beliaev2019two}
\begin{align*}
\mathbb{E}(\mathcal N_{\rho }^{c}(\mathcal  N_{\rho }^{c}-1) )\sim \frac{ k^{4}}{2^{6}3\sqrt{3}}\rho ^{4}.
\end{align*}
\end{ex}

\begin{ex}\label{ex:Ft}Consider the centered stationary Gaussian random field $\psi $ with spectral measure 
\begin{align*}
F(d\lambda )= | \lambda |  ^{-7}\mathbf{1}_{\{ | \lambda  | \geqslant 1\}}d\lambda .
\end{align*}
One has by Proposition \ref{prop covari et variance en deriv},$
 | \sigma '(0) |  <\infty,
\sigma ''(0)<\infty ,
-\sigma '''(0)=\infty , $
hence $\mathsf R_{c}=\infty $, but  Theorem \ref{theom second 2 crtique critque and extrema and s} does not apply precisely because $F$'s higher moments are infinite, meaning that $\psi $ is not of class $\mathcal  C^{3}$. Hence we consider $F_{t}(d\lambda )=\mathbf{1}_{\{ | \lambda  | \leqslant t\}}F(d\lambda )(\int_{ 0}^{t}dF )^{-1}$ for $t>1$. We have as $t\to \infty $
\begin{align*}
\frac{ \eta _{0}\nu _{0}}{\mu _{0}^{2}}\sim c\int_{ 1}^{t} r^{6}r^{-7}rdr\asymp t.
\end{align*}
It implies that the repulsion factor of $F_{t}$ can reach arbitrarily high values. In particular, this parametric model provides processes of critical points that are weakly locally attractive.
\end{ex}

\subsection{Discussion and related litterature}

The equivalence \eqref{result of the second factorial moment critical points theroem} generalises the results of \cite{beliaev2019two}, and shows that locally, the random planar wave model yields the more repulsive critical points. We also show that for a general process $\psi $, the subprocesses formed by extrema and saddle points experience locally a strong repulsion with three more orders of magnitude for $\rho $.  It confirms the idea that close to a large portion of  saddle points, there is an extremal point nearby, and conversely, but that the closest point of the same type (extrema or saddle) is typically  much further away. 

A current novelty is also to derive the precise asymptotic repulsion for the extrema process and the saddle process. Hence we are able to state that  the ratio between the internal repulsion forces among extremal points and among saddle points tends to infinity as the radius of the observation ball goes to $0$.

{Azais and Delmas \cite{azais2019mean} derived upper bounds about such quantities in any dimension. In particular, their results are consistant with ours in the critic-critic, extrema-extrema and extrema-saddle cases. }

\section{Proofs}

\subsection{Conditioning}
The proofs of  all formulas of  Theorem \ref{theom second 2 crtique critque and extrema and s} are  based on  the Kac-Rice formula in Theorem \ref{thm:KR}, for instance if $B=B'=\mathbb{R}^{2}$, we have the second factorial moment of the number of critical points\begin{equation}
\label{seoncd facotirle moment of numbr of criticla points}
\mathbb{E} \left[ \mathcal{N}_{\rho}^{c}  (\mathcal{N}_{\rho}^{c} -1) \right]=\int\int_{\B_\rho \times \B_\rho} \mathrm{K}_{2}(z,w) \; dz dw,
\end{equation}
where $\mathrm{K}_{2}$ is the $2$-point correlation function : :
\begin{align}
\label{the 2 point untion crituq critique }
\mathrm{K}_{2}(z,w)&=\phi_{(\nabla \psi (z) , \nabla \psi (w))}((0,0),(0,0))\\
&\times   \mathbb{E} \left[ \;    \large|\det H_{\psi}(z)\large|\;  \large| \det H_{\psi}(w) \large|  \; \big| \; \nabla \psi (z) =\nabla \psi (w)=0  \; \right]\nonumber .
\end{align}

 {Let us briefly introduce where the difficulty comes from and why higher order differentiability is required. For $z,w$ close from $0$, if $\nabla \psi (z)=\nabla \psi (w)=0$, then the second order derivatives are also small, and the determinant  is dominated by third order differentials. When one imposes additional constraints on the determinant signs, it yields other cancellations within third order derivatives, requiring fourth order differentiability.}

Thanks  to the stationarity and isotropy of $\psi  ,$ it suffices to  compute $\mathrm{K}_{2}(z,w)$ for $z=(r,0)$  and $w=(-r,0)$ for all $r>0.$ 
To evaluate $\mathbb{E} \left[ \mathcal{N}_{\rho}^{c}  (\mathcal{N}_{\rho}^{c} -1) \right]$, the idea is to change the conditioning in $ \mathrm{K}_{2}(z,w).$  To symmetrize  the problem,  we introduce  some notations 
for $r$ near  $0$, $r  \neq 0$, exploiting Proposition \ref{prop:C4alpha} and Definition \ref{def:OP}, 
 \begin{equation}
\label{notation de deltai deltai1}
\left\{ \begin{array}{lll}
a)  &  \Delta_{i}(r):=\frac{1}{2}\partial_{i}\psi (z )+\frac{1}{2}\partial_{i}\psi (w)  & \mbox{  implies  }\Delta_{i}(r)= \partial_{i}\psi(0) + \frac{r^2}{2} \partial_{i11}\psi(0)+\Op(r^4) \\
b) &  \Delta_{i1}(r): =\frac{1}{2 r } (\partial_{i}\psi (z)-\partial_{i}\psi (w ) )&  \mbox{  implies  }  \Delta_{i1}(r)=\partial_{i1}\psi(0) + \frac{r^2}{6} \partial_{i111}\psi(0)+\Op(r^4).\\
\end{array}\right.,\\
\end{equation} 

The crucial point is that $\nabla \psi (z)=\nabla \psi (w)=0$ is equivalent to $\Delta_{i}(r)=0,\Delta_{i1}(r)=0,i=1,2.$
 Let us introduce $$Y_{r}=( \Delta_{1}(r), \Delta_{2}(r), \Delta_{11}(r), \Delta_{12}(r)),r>0$$ so that $Y_{r}=0$ is equivalent to $\nabla \psi (z)=\nabla \psi (w)=0$, and $$Y_{0}:=(\partial_{1}\psi(0),  \partial_{2}\psi(0), \partial_{11}\psi(0),  \partial_{21}\psi(0)).$$ We will see later that $Y_{0}$ is non-degenerate, hence $Y_{r}$ is also non-degenerate for $r$ small enough, by continuity of the covariance matrix.\\

We denote the conditionnal probability and expectation with respect to $Y_{r}=0$ by 
\begin{align*}
\mathbb{P}^{(r)}(\cdot )=\mathbb{P}(\cdot \;|\;Y_{r}=0),\qquad \mathbb{E}^{(r)}(\cdot )=\mathbb{E}(\cdot \;|\;Y_{r}=0),r\geqslant 0.
\end{align*}

\begin{remark}
Let $(X,Y)$ be a Gaussian vector with $Y$ non-degenerate. If $M$ is non-singular matrix and if $\varphi$ is a  measurable function with polynomial bounds, then 
\begin{equation*}
E(\varphi(X)|Y=0)=E(\varphi(X)|MY=0).
\end{equation*} 
\end{remark}

So, since obviously $$\nabla\psi (z)=\nabla\psi (w)=0   \iff  Y_{r}=0,  \; $$
the  2-point  correlation function ${\mathrm{K}}_{2}(z,w)$  becomes :
\begin{align}
\label{the 2 point untion crituq critique 2 }
\mathrm{K}_{2}(z,w)=\phi_{(\nabla \psi (z) , \nabla \psi (w))}((0,0),(0,0))\times   \mathbb{E} ^{(r)}[ \;  | \det H_{\psi}(z)|\;  | \det H_{\psi}(w)|  ]  .
\end{align}\\
Using \eqref{eq:det-cov-grad} we can evaluate the density in $0$, and the previous expression becomes 
\begin{align*}
K_{2}(z,w)=\frac{ 1}{(2\pi )^{2}2^5 \sqrt{3} ( - \eta_{0} )    r ^{2}} \mathbb{E} ^{(r)}[ \;  | \det H_{\psi}\psi(0)|\;  | \det H_{\psi}(r)|  ](1+o(1)).
\end{align*}

It remains to express the product of determinants   under the conditioning in function of $\Delta _{i}(r)=\Delta _{i1}(r)=0$, this involves higher order derivatives (see\eqref{notation de deltai deltai1}).
\begin{lemma}
	\label{produit de determin en point z et w forme felgen section }
	Assume $\psi  $ satisfies Assumption \ref{ass:1} for some $\beta >0$ and let $0<\alpha <\beta /2$.
	For $z=(r,0)$ and $w=(-r,0)$, we have {if $\Delta _{i}=\Delta _{1i}=0,i=1,2,$}
	\begin{align}
	\label{produit hesz hesw}
	\det H_{\psi }(z)=& r(A_{1}+rB_{0}+\Op(r^{1+\alpha }))\\
	\det H_{\psi }(w)=&r(-A_{1}+rB_{0}+\Op(r^{1+\alpha }))\\
	\det  H_{\psi}(z) \det  H_{\psi}(w)=&r^2 [- {A_{1}}^{2}+g(r) ]
	\end{align}
	where  
	\begin{equation}
	\label{ Aet valeur B ET B' et C ET C'}
	\left\{ \begin{array}{lll} &	A_{1}= \partial_{22}\psi(0)\partial_{111}\psi(0) \\ 
	&B_{0}= \partial_{221}\psi(0)\;  \partial_{111}\psi(0)-\partial_{211}\psi(0)^2 +  \frac{1}{3} \partial_{22}\psi(0)\partial _{1111}\psi(0)\\
	&g(r)=r^{2}B_{0}^{2}+\Op(A_{1}r^{1+\alpha })+\Op(r^{2+\alpha })
	\end{array}\right. .
	\end{equation}
\end{lemma}

\begin{proof} 
Define

	\begin{equation*}
	\left\{ \begin{array}{lll}	a)  &\frac{r^2}{3} \Delta_{1111}(\pm r) =\partial_{11}\psi ( \pm r,0 )-\Delta_{11}( \pm r) -(\pm  \; r \partial_{111}\psi(0))  &  \mbox{  implies   }  \Delta_{1111}( \pm r) =\partial_{1111}\psi(0) +\Op(r^{\alpha }) \\
	b) &\frac{r^2}{3}\Delta_{2111}(\pm r)=\partial_{21}\psi ( \pm r,0 )-\Delta_{21}(\pm r)  -(\pm  \; r \partial_{211})  &  \mbox{  implies }  \Delta_{2111}(\pm r) =\partial_{2111}\psi(0)+\Op(r^{\alpha }) \\
	c) & \frac{r^2}{2} \Delta_{2211}(\pm r) = \partial_{22}\psi ( \pm r,0 )-\partial_{22}\psi(0)-( \pm  \; r \partial_{221} )  &  \mbox{ implies}  \;  \Delta_{2211}(\pm r ) = \partial_{2211}\psi(0) +\Op(r^{\alpha }). \\
	\end{array}\right.
	\end{equation*}
	We can explicitly write the expression of  $\det  H_{\psi}(z) \det  H_{\psi}(w):$  
	\begin{align}
	\label{detr hess au point z}
	\det  H_{\psi}(z)&=  \partial_{11}\psi ( r,0 )  \partial_{22}\psi (r,0 )- (  \partial_{21}\psi (r,0 ))^2\nonumber \\
	&=r(A_{1}+rB_{r}+r^2 C_{r}),
	\end{align}with
	 	
\begin{align*}
B_{r}&=  \partial_{111}\psi(0)  \partial_{221}\psi(0)   + \frac{1}{3} \Delta_{1111}(r) \; \partial_{22}\psi(0)-  \partial_{211}\psi(0) ^2 =B_{0}+\Op(r^{\alpha })\\
C_{r}&=\frac{\Delta_{1111}(r)}{3} \; \partial_{221} \psi(0) + \frac{ \Delta_{2211}(r) }{2}  \partial_{111}\psi(0)  - \frac{2}{3}  \Delta_{2111}(r) \partial_{211}\psi(0)=\Op(1) 
\end{align*}
and
	\begin{align}
	\label{detreminant hess au point w}
	\det  H_{\psi}(w)&=  \partial_{11}\psi ( -r,0 )  \partial_{22}\psi (-r,0 )- (  \partial_{21}\psi (-r,0 ))^2\nonumber \\
	& {=r(-A_{1}+rB_{r}'+r^{2}C_{r}')}
	\end{align} with
\begin{align*}
B_{r'}=& \partial_{111}\psi(0)  \partial_{221} \psi(0)  + \frac{1}{3} \Delta_{1111}(-r) \; \partial_{22}\psi(0) -  \partial_{211}\psi(0) ^2=B_{0}+\Op( r^{\alpha })\\
C_{r}'&=-\frac{\Delta_{1111}(-r)}{3} \; \partial_{221}\psi(0)  - \frac{ \Delta_{2211}(-r) }{2}  \partial_{111}\psi(0)  + \frac{2}{3}  \Delta_{2111}(-r) \partial_{211}\psi(0) =\Op(1)
\end{align*}
	
	Combining Equations \eqref{detr hess au point z} and \eqref{detreminant hess au point w}, an elementary calculus  leads to 
	\begin{equation*}
	\det  H_{\psi}(z) \det  H_{\psi}(w)=r^2 [- {A_{1}}^{2}+g(r) ]
	\end{equation*} 
	 where 
		\begin{align*}
		g(r)=& r A_{1} (-B_{r} +B_{r'})+r^{2}B_{r} B_{r'}-r^{2} A_{1} (C_{r}- C_{r}')\\
		&= r A_{1} (\Op( r^{\alpha }))+r^{2}B_{r}B_{r'}-r^{2} A_{1} (C_{r}- C_{r}')\\
		&=\Op(A_{1}r^{1+\alpha }) +r^{2}B_{0}^{2}+\Op(A_{1}r^{1+\alpha })+\Op(r^{2+\alpha })+\Op(r^{2+2\alpha })\\
		&=r^{2}B_{0}^{2}+\Op(A_{1}r^{1+\alpha })+\Op(r^{2+\alpha }).
		\end{align*}

\end{proof}

As a consequence from Lemma \ref{produit de determin en point z et w forme felgen section }, the  2-point  correlation function given by \eqref{the 2 point untion crituq critique 2 } becomes 
\begin{align}
\label{the 2 point untion crituq critique 3 }
\mathrm{K}_{2}(z,w)&=\frac{ 1}{(2\pi )^{2}2^5 \sqrt{3}(- \eta_{0} )\mu_{0}  }  \mathbb{E}^{(r)} \left[ 
| {A_{1}}^{2}-g(r)|   \right]   (1+o(1)).
\end{align}

\subsection{Dependency of derivatives}
\label{sec:dependency}
In view of the previous result, we  will have to estimate quantities related to the random vectors $$X=(\partial _{22}\psi(0),\partial _{111}\psi(0),\partial _{122}\psi(0),\partial _{112}\psi(0),\partial _{1111}\psi(0))$$ and $Y_{r}$. We  must consider the case where $(X,Y_{0})$ is degenerate. Examining Remark \ref{rk:indep-deriv}, we can split the variables involved in several groups that are mutually independent, there are for instance only two groups of size $3$, 
\begin{align*}
\{\partial _{1}\psi (0),\partial _{122}\psi (0),\partial _{111}\psi (0)\}\text{\rm{  and  }}\{\partial _{22}\psi (0),\partial _{1111}\psi (0),\partial _{11}\psi (0)\}.
\end{align*}Other groups, such as $\{{\partial _{112}\psi (0),\partial _{2}\psi (0)}\}$, have less members, and in the isotropic case they won't be in a linear relation because of \eqref{eq:CS-lambda}:
\begin{align*}
\text{\rm{ Cov }}(\partial _{112}\psi (0),\partial _{2}\psi (0))^{2}=m_{2,2}^{2}<m_{2,0}m_{4,2}=\var(\partial _{2})\var(\partial _{112}).
\end{align*}
There is actually no other case to consider.
	 Let us elucidate what can happen within the two bigger groups.

\begin{proposition}
\label{eq:edp-rw}
Assume the spectral measure $F$ is isotropic and not reduced to a Dirac mass in $0$. There is $(\alpha ,\beta ,\gamma) \neq (0,0,0)$ such that $\alpha \partial _{1}\psi (0)+\beta \partial _{111}\psi (0)+\gamma \partial _{122}\psi (0)=0$ a.s. iff $\beta =\gamma $ and $F $ is uniformly spread along a circle of radius $\sqrt{\alpha /\beta }$, i.e. if $\psi $ is a SGRW with parameter $\sqrt{\alpha  /\beta  }$.

There is no $(\alpha ,\beta ,\gamma) \neq (0,0,0)$ such that a.s. $\alpha \partial _{11}\psi (0)+\beta \partial _{22}\psi (0)+\gamma \partial _{1111}\psi (0)=0$.
\end{proposition}

\begin{proof}
Using \eqref{eq:spectral-deriv} and recalling the symmetry $m_{a,b}=m_{b,a}$
\begin{align*}
\var(\alpha \partial _{1}\psi (0)+\beta \partial _{111}\psi (0)+\gamma \partial _{122}\psi (0))=&\alpha ^{2}m_{2,0}+\beta ^{2}m_{6,0}+\gamma ^{2}m_{2,4}+2\alpha \beta m_{4,0}+2\alpha \gamma m_{2,2}+2\beta \gamma m_{4,2}\\
=&\int_{}(-\alpha^{2} \lambda _{1}^{2}-\beta ^{2}\lambda _{1}^{6}-\gamma ^{2}\lambda _{1}^{2}\lambda _{2}^{4}+2\alpha \beta \lambda _{1}^{4}+2\alpha \gamma \lambda _{1}^{2}\lambda _{2}^{2}-2\beta \gamma \lambda _{1}^{4}\lambda _{2}^{2} )F (d\lambda )\\
=&-\int_{}(-\alpha \lambda _{1}+\beta \lambda _{1}^{3}+\gamma\lambda _{1} \lambda _{2}^{2} )^{2}F (d\lambda ).
\end{align*}
Hence, $dF$-a.s., either $\lambda _{1}=0$ or $\gamma \lambda _{1}^{2}+\beta \lambda _{1}^{2}=\alpha $. By isotropy, it implies that $\gamma =\beta $ and that $F$'s support is concentrated on zero and the circle with radius $\sqrt{\alpha/\beta  }$. It corresponds to the GRW with radius $\sqrt{\alpha /\beta }$ plus an additional constant term.

In the same way, 
\begin{align*}
0=\var(\alpha \partial _{22}\psi (0)+\beta \partial _{11}\psi (0)+\gamma \partial _{1111}\psi (0))=\int_{}(\alpha \lambda _{2}^{2}+\beta \lambda _{1}^{2}+\gamma \lambda _{1}^{4})^{2}F (d\lambda )
\end{align*}
implies that $F$ is trivial if $F$ is isotropic.
\end{proof}

In conclusion, the only non-trivial linear relations possibly satisfied by the derivatives involved in $(X,Y_{0})$ is $\partial _{111}\psi (0)+\partial _{122}\psi (0)=\alpha \partial _{1}\psi (0),\alpha >0$ and can only be satisfied by a  SGRW. In the light of these results, functionals of interest only depend on the law of the vector $X'$ under the conditioning $Y_{0}=0$, where 
\begin{align*}
X'=\begin{cases}(\partial _{22}\psi(0),\partial _{111}\psi(0),\partial _{112}\psi(0),\partial _{1111}\psi(0))$ if $\psi $ is a shifted GRW$\\
X$ otherwise$ \end{cases}
\end{align*}
because if $\psi $ is a shifted GRW and $Y_{0}=0$, $\partial _{111}\psi (0)+\partial _{122}\psi (0)=-\alpha \partial _{1}\psi (0)=0$ a.s. hence $\partial _{122}\psi (0)$ is directly expressible in function of $\partial _{111}\psi (0).$

\begin{lemma}
\label{lm:indep-density}
The conditional density $f_{r}$ of $X'$ knowing $Y_{r}$ converges pointwise to the density $f_{0}$ of $X'$ knowing $Y_{0}$. There is furthermore $\sigma _{1},\sigma _{2},c_{1},c_{2}>0$ such that for $r$ sufficiently small, 
\begin{align*}
c_{1}g_{\sigma _{1}}\leqslant f_{r}\leqslant c_{2}g_{\sigma _{2}}
\end{align*}
where $g_{\sigma }$ is the density of iid Gaussian variables $Z^{\sigma }=(Z^{\sigma }_{i})_{i}$ with common variance $\sigma ^{2}.$ Hence for any non-negative functional $\varphi _{r}$
\begin{align*}
c_{1}\mathbb{E}(\varphi _{r}(\sigma _{1}Z^{1}))\leqslant \mathbb{E}^{(r)}(\varphi _{r}(X'))\leqslant c_{2}\mathbb{E}(\varphi _{r}(\sigma _{2}Z^{1})).
\end{align*}
\end{lemma}

\begin{proof}

Since the vector $(X',Y_{0})$ is non-degenerate, by continuity of the covariance matrix, the vector $(X',Y_{r})$ is non-degenerate either for $r$ sufficiently small, and the density of the former converges pointwise to the density of the latter. Hence the conditionnal  density  $f _{r}$ of $(X'\; | \; Y_{r})$ converges to  $f_{0}$ the non-degenerate multivariate conditional Gaussian density of $(X\; | \; Y_{0})$.\\
 Let  $\Gamma_{r}$ be the  covariance matrix of the conditional vector  $(X'\; | \; Y_{r}).$ 	For $ 1\leqslant i \leqslant d,$ we denote by $\lambda_{i ,}(r)$ the $i$-th eigenvalue of the matrix  $\Gamma_{{r}}$. Since $\lambda _{i}(r)\to \lambda _{i}(0)>0$, there exists constant $\sigma _{1}\geqslant \sigma _{2}>0$  such that for $r $ sufficiently small
	\begin{align} 
	\label{minore les valeurs propres }
	\frac{ 1}{2\pi \sigma _{2}^{2}}\geqslant \lambda_{i }(r) \geqslant \frac{ 1}{2\pi \sigma _{1}^{2}}.
	\end{align}
Hence $f_{r}(x)$ is  bounded between $c\exp(-\sum_{i}x_{i}^{2}/(2\pi \sigma _{1}^{2}))$ and $c'\exp(-\sum_{i}x_{i}^{2}/(2\pi \sigma _{2}^{2}))$ for some $c,c'>0$, which gives the desired claims.
 \end{proof}

\subsection{ Proof of \eqref{result of the second factorial moment critical points theroem} in Theorem \ref{theom second 2 crtique critque and extrema and s} }

From \eqref{the 2 point untion crituq critique 3 },  we have 
	\begin{align*}
	{\mathrm{K}}_{2}(z,w)&=  \frac{ r^{2}}{(2\pi )^{2}2^5 \sqrt{3} (- \eta_{0})     \mu_{0}   (1+o(1))}  \mathbb{E}^{(r)} \left[ 
| {A_{1}}^{2}-g(r)|   \right] .
	\end{align*}
	
According to Lemma \ref{lm:indep-density}, the conditional density $f_{r}$ of $X'$ knowing $Y_{r}$ converges pointwise to the non-degenerate density $f_{0}$ of $X'$ knowing $Y_{0}$, and $\varphi _{r}(X'):=A_{1}^{2}-g(r)$ is uniformly bounded by a polynomial  $P (X')$, Lebesgue's Theorem then yields  \begin{align}
	\label{limite de varphi A1 au carre}
	\lim\limits_{r \rightarrow 0} \mathbb{E}^{{(r)}}  \left[  \mid  A_{1}^2 - g(r) \mid    \right]&=\int_{ }\varphi _{r}(x)f_{r}(x)dx\to \int_{ }\lim_{r}\varphi _{r}(x)f_{0}(x)dx=    \mathbb{E}^{(0)}  \left[   A_{1}^2 \right].
	\end{align} 
	
 To compute the conditionnal law of $Z:=(\partial _{22}\psi(0),\partial _{111}\psi(0))$, recall that  in virtue of \eqref{eq:spectral-deriv} and Remark \ref{rk:indep-deriv}  $\partial _{22}\psi(0)$ and $\partial _{111}\psi(0)$ are independent, and the covariance matrix of $Y_{0}$ is $$\Gamma (Y_{0})=\left(\begin{array}{cccc}m_{2,0} & 0 & 0 & 0 \\0 & m_{2,0} & 0 & 0 \\0 & 0 & m_{4,0} &0\\0 & 0 & 0 & m_{2,2}\end{array}\right).$$ The covariance matrix of $Z$ and $Y_{0}$ is 
\begin{align*}
\Gamma( {Z,Y_{0}})=\left(\begin{array}{cccc}m_{2,1} & m_{0,3} & m_{2,2} & m_{1,3} \\m_{4,0} & m_{3,1} & m_{5,0} & m_{4,1} \end{array}\right)=\left(\begin{array}{cccc}0 & 0 & m_{2,2} & 0\\m_{4,0} & 0 & 0 & 0\end{array}\right).
\end{align*}
It follows that the conditional covariance of $Z$ knowing $Y_{0}$ is  
\begin{align*}
\Gamma ({Z\;|\;Y_{0}})=\Gamma (Z)-\Gamma (Z,Y_{0})\Gamma (Y_{0})^{-1}\Gamma (Z,Y_{0})^{t}=
&\left(\begin{array}{cc}m_{0,4}-m_{4,0}^{-1}m_{22}^2 & 0 \\0 & m_{6,0}-m_{2,0}^{-1}m_{4,0}^2\end{array}\right)\\
&=\left(\begin{array}{cc}\var(\partial_{22}\psi(0)|\partial_{111}\psi(0)) & 0 \\0 & \var(\partial_{111}\psi(0)|\partial_{1}\psi(0))\end{array}\right)
\end{align*} 
the diagonal terms are positive in virtue of \eqref{eq:CS-lambda}.
By conditionnal independence of $\partial _{22}\psi(0)$ and $\partial_{111}\psi(0)$, and using Proposition \ref{prop covari et variance en deriv}
\begin{align*}
\mathbb{E}^{(0)}(A_{1}^{2})=\var(\partial_{22}\psi(0)|\partial_{11,1}\psi(0)) \var(\partial_{11,1}\psi(0)|\partial_{1}\psi(0))&=(m_{0,4}-m_{4,0}^{-1}m_{22}^2)  (m_{6,0}-m_{2,0}^{-1}m_{4,0}^2)>0\\
&=\frac{2^5}{3} \mu_{0} (\frac{3 \cdot 2^3 }{\eta_{0}} (-5 \nu_{0} \eta_{0}+ 3 \mu_{0}^2))\\
&={2^8}\mu_{0} \frac{  (-5 \nu_{0} \eta_{0}+ 3 \mu_{0}^2)}{\eta_{0}}.
\end{align*}

	Combining  Equation \eqref{limite de varphi A1 au carre} , \eqref{the 2 point untion crituq critique 3 } we obtain 
\begin{align}
\label{calcul de K 2 limite dans criti crti}
\lim\limits_{r \rightarrow 0}  {\mathrm{K}}_{2}(z,w)=\frac{ 1}{(2\pi )^{2} \;2^5 \sqrt{3} (-\eta_{0}) \mu_{0}   (1+o(1))} \; {2^8}\mu_{0} \frac{  (-5 \nu_{0} \eta_{0}+ 3 \mu_{0}^2)}{\eta_{0}}=\frac{ 10\;     \nu_{0}  \eta_{0} -6 \mu_{0}^2  }{\pi ^{2} \;  \sqrt{3} \eta_{0}^{2}     (1+o(1))} =:a(1+o(1)) .
\end{align}

Finally, the second factorial moment of  $\mathcal{N}^{c}_{\rho}$ when $\rho \rightarrow 0, $ is given by 
\begin{align*}
\mathbb{E}[ \mathcal{N}^{c}_{\rho} (\mathcal{N}^{c}_{\rho}-1)] &= \int    \int_{\B_\rho\times   \B_\rho}{\mathrm{K}}_{2}(z,w) \,\mathrm{d}z \; \mathrm{d}w\nonumber=a  | B_{\rho } | ^{2}(1+o(1)).
\end{align*}
Recalling that 
\begin{align*}
\lambda _{c} =\frac{4 }{\sqrt{3}} \frac{\mu _{0} }{-\eta _{0}\pi }
\end{align*}
yields indeed $a=\lambda _{c}^{2}\mathsf R_{c}.$

Let us show that $\mathsf R_{c}\geqslant \frac{ 1}{8\sqrt{3}}$. Given a measure $\mu $ on $\mathbb{R}_{+}$, denote by $$m_{k}(\mu )=\int_{ }t^{k}\mu (dt). $$  Since $F$ is isotropic, define $\mu $ as the radial part of $F$, yielding with a polar change of coordinates
\begin{align*}
\int_{\mathbb{R}^{2} }\lambda _{1}^{a}\lambda _{2}^{b}F(d\lambda )= {\int_{ 0}^{2\pi } \cos(\theta )^{a}\sin(\theta )^{b}} d\theta m_{a+b+1}(\mu ).
\end{align*}
Introduce the probability measure, for $A\subset \mathbb{R}_{+},$ $$\tilde \mu (A)=\frac{ \int_{ A}t\mu (dt) }{m_{1}(\mu )}.$$ Using the spectral representation in Proposition \ref{prop covari et variance en deriv} yields for some $c>0$
\begin{align*}
\frac{ \nu _{0}\eta _{0}}{\mu _{0}^{2}}=&c\frac{ m_{3}(\mu )m_{1}(\mu )}{m^{2}(\mu )^{2}}=c\frac{m_{1}(\mu ) m_{2}(\tilde \mu )m_{1}(\mu )}{(m_{1}(\mu )m_{1}(\tilde \mu ))^{2}}=c\frac{ m_{2}(\tilde \mu )}{m_{1}(\tilde \mu )^{2}}\geqslant c
\end{align*}
by the Cauchy-Schwarz inequality. The ratio is minimal if the equality is obtained in the Cauchy-Schwarz inequality, i.e.  when $t^{2}$ is proportionnal to $t$ $\tilde \mu $-a.e.. This is the case only if $F(d\lambda ) $ is uniformly spread on a circle of $\mathbb{R}^{2}$, with perhaps also an additional atom in $0$. This corresponds exactly to the class of fields derived in Example \ref{ex:grw}, which are the SGRW. For the precise computation of the constant $\frac{ 1}{8\sqrt{3}}$, see Example \ref{ex:factorial-grw}.

In example \ref{ex:Ft}, we derive spectral measures $F_{t},t> 1$ which achieve repulsion factors $\mathsf R_{c}$ in an interval of the form $(\alpha _{0},\infty )$ for some $\alpha _{0}\in \mathbb{R}$. Therefore it remains to show that all values between $\frac{ 1}{8\sqrt{3}}$ and $\alpha _{0}$ can be achieved. For that we use an interpolation 
\begin{align*}
G_{s}:=sF_{RW}+(1-s)F_{2},s\in [0,1]
\end{align*}
where $F_{RW}$ is the spectral measure of a GRW and $F_{2}$ belongs to the parametric family $F_{t},t\geqslant 1$. The ratio of moments
\begin{align*}
s\mapsto \frac{ m_{3}(G_{s})m_{1}(G_{s} )}{m^{2}(G_{s})^{2}}
\end{align*}
evolves continuously with $s$ because all the members of the numerator and denominator do, hence the repulsion factor evolves continuously between $\frac{ 1}{8\sqrt{3}}$ and $\alpha _{0}$ and achieves all intermediary values.

\subsection{  Proof of \eqref{result od the second moment of extrema points } in Theorem \ref{theom second 2 crtique critque and extrema and s}} 
To compute  the second factorial moment of $\mathcal{N}^{e}_{\rho}=\mathcal  N_{\rho }^{(0,\infty )}$ when $\rho \rightarrow 0$,
we apply  the Kac-rice formula of Theorem \ref{thm:KR} in the case $B_{1}=B_{2}=(0,\infty )$
\begin{equation}
\label{the facotroial moment kac rice cas general  extrema}
\mathbb{E}[ \mathcal{N}^{e}_{\rho} (\mathcal{N}^{e}_{\rho}-1)]= \int    \int_{\B_\rho\times   \B_\rho}  {\mathrm{{K}}}^{e,e}_{2}(z,w) \,\mathrm{d}z \; \mathrm{d}w, 
\end{equation}
where   
\begin{align*}
{\mathrm{K}}^{e,e}_{2}(z,w)=&  \;  \phi_{(\nabla \psi (z),\nabla \psi (w)  ) }((0,0)), (0,0)) \\
&\times   \mathbb{E}^{(r)} \left[  \; | \det H_{\psi}(z)|\;  | \det H_{\psi}(w)|  \; \mathbf{1}_{\{ \det H_{\psi}(z)>0 \}} \;  \; \mathbf{1}_{\{ \det H_{\psi}(w)>0\}}  \right].
\end{align*} It becomes in virtue of \eqref{produit hesz hesw}
\begin{equation}
\label{kEE exrep extremza extrema }
{\mathrm{K}}^{e,e}_{2}(z,w)=    r^2  \; \phi_{(\nabla \psi (z),\nabla \psi (w)  )}((0,0),(0,0))  \; a_{r} \\
\end{equation}
where 
\begin{align*}
a_{r}:=&\mathbb{E}^{(r)} \left[  \; \big|  {A_{1}}^{2}-g(r)  \big| \;   I_{r}   \right]\\
I_{r}:=& \mathbf{1}_{\{ \det H_{\psi}(z)>0 \}}  \mathbf{1}_{\{ \det H_{\psi}(w)>0\}}.
\end{align*}
To be able  to prove  \eqref{result od the second moment of extrema points }, we need to establish  an upper  bound and a lower bound    of $a_{r}$
So the proof is  separated into two parts. We first give in Lemma \ref{asmpy expression of ar} an asymptotic expression of $a_{r}$  to get rid of superfluous variables.
\begin{lemma}
\label{lm:ar}
Let $J_{r}:=\mathbf{1}_{\{ | A_{1} | <rB_{0}\}}$,
	\label{asmpy expression of ar}
	\begin{align*}
	|a_{r}-\mathbb{E}^{(r)}(|A_{1}^{2}-r^{2}B_{0}^{2}| J_{r}
	)|=\Op(r^{3+\alpha' })
	\end{align*}for $0<\alpha '<\alpha ,$
	and as $r\to 0$
	\begin{align}
	\label{eq:equiv-ar}
a_{r}\asymp \mathbb{E}( | A_{1}^{2}-r^{2}B_{0}^{2} | J_{r}).
\end{align}
\end{lemma}

\begin{proof}

From \eqref{detr hess au point z}-\eqref{detreminant hess au point w} in the proof of Lemma  \ref{produit de determin en point z et w forme felgen section },
\begin{align*}
I_{r}\leqslant \mathbf{1}_{ \{| A_{1} | <rD_{r}\}}
\end{align*}
where 
\begin{align*}
D_{r}:=& | B_{0} | +| {B}_{r} |+ | B_{r} '|   + r ( |C_{r}|+ |C'_{r}|),
\end{align*}
  is a variable with Gaussian tail.
	 Recall that $ g(r)= r^2 B_{0}^{2	}+\Op(A_{r}r^{1+\alpha }+r^{2+\alpha })$, hence using \eqref{ Aet valeur B ET B' et C ET C'},
	 \begin{align}
	 \notag
	\mathbb{E}^{(r)}
	\left[(| A_{1}^{2}-g( r) | - | A_{1}^{2}-r^{2}B_{0}^{2} |
	)I_{r})\right]\leqslant &\mathbb{E}^{(r)}(\Op(r^{1+\alpha }A_{1}+r^{2+\alpha }))|\mathbf{1}_{\{ | A_{1} | <r D_{r}   \}} )\\
	 \label{eq:666}
	\leqslant & r^{2+\alpha }\mathbb{E}^{(r)}(\Op(D_{r}+1)\mathbf{1}_{\{ | A_{1} | <rD_{r}\}}).
	\end{align}
	Let $p,q> 1,\eta >0$ such that $p^{-1}+q^{-1}=1$ and $\alpha +\frac{1- \eta }{q}> \alpha '+1$, then Holder's inequality yields
\begin{align*}
\mathbb{E}^{(r)}(\Op(D_{r}+1)\mathbf{1}_{\{ | A_{1} | <rD_{r}\}})\leqslant \mathbb{E}^{(r)}(\Op(D_{r}+1)^{p})^{\frac{ 1}{p}}\mathbb{P}^{(r)}( | A_{1} | <rD_{r})^{\frac{ 1}{q}}.
\end{align*}The probability on the right hand member can be bounded by
\begin{align*}
\mathbb{P}^{(r)}( | A_{1} | <rD_{r})\leqslant \mathbb{P}^{(r)}(  | D_{r} |  >r^{-\eta  })+\mathbb{P}^{(r)}( | A_{1} | <r^{1-\eta }).
\end{align*}
All variables involved in $\Op(D_{r})$ have a Gaussian tail, hence 
\begin{align*}
\mathbb{P}^{(r)}(\Op( | D_{r} |) >r^{-\eta  })=o(r^{2}).
\end{align*}
By Lemma \ref{lm:indep-density}  with $\varphi _{r}(x)=\mathbf{1}_{\{ | x_{1}x_{2} | <r^{1-\eta }\}},$
and Lemma \ref{lm:ineg-produits}-(i) (with $s=0$), 
\begin{align}
\label{eq:11}
\mathbb{P}^{(r)}( | A_{1} | <r^{1-\eta })=\mathbb{E}^{(r)}(\varphi _{r}(X'))\leqslant  c_{2}\mathbb{E}(\varphi _{r}(\sigma _{2}Z)) <&c'r^{ { 1-\eta }  }\ln(r) 
\end{align}
hence finally
\begin{align}
\label{eq:667}
r^{2+\alpha }\mathbb{E}^{(r)}(\Op(D_{r}+1)\mathbf{1}_{\{ | A_{1} | <rD_{r}\}})<&c'r^{2+\alpha + \frac{ 1-\eta }{q}}\ln(r)^{1/q}=O(r^{3+\alpha '}).
\end{align}
		
	To simplify indicators, remark that in virtue of \eqref{detr hess au point z},\eqref{detreminant hess au point w},
	\begin{align*}
	I_{r}&=\mathbf{1}_{\{A_{1}+rB_{r}+r^{2}C_{r}>0,-A_{1}+rB_{r}'+r^{2}C_{r}''>0\}}\\
	J_r &=\mathbf{1}_{\{A_{1}+rB_{0}>0,-A_{1}+rB_{0}>0\}}.	\end{align*} 
	Both the events  $\{I_{r}=1\},\{J_{r}=1\}$ imply $ | A_{1} | <rD_{r}$. If $I_{r}\neq J_r $,   $A_{1}+rB_{r}+r^{2}C_{r}$ has a sign different from $A_{1}+rB_{0}$, or $-A_{1}+rB_{r}'+r^{2}C_{r}'$ has a sign different from $-A_{1}+rB_{0}$. In both cases it implies another event of magnitude $\Op(r^{1+\alpha })$ because $B_{r},B_{r}'=B_{0}+\Op(r^{\alpha }), C_{r},C_{r}'=\Op(1):$
	\begin{align*}
	|I_{r}-J_r |&\leqslant 2 \left(\mathbf{1}_{\{( A_{1}+rB_{r}+r^{2}C_{r})\cdot ( A_{1}+rB_{0})<0\}}+\mathbf{1}_{\{( -A_{1}+rB'_{r}+r^{2}C_{r}')\cdot ( -A_{1}+rB_{0})<0\}}\right)\\&\leqslant 2\left(\mathbf{1}_{\{ | A_{1}+rB_{0}|<\Op(r^{1+\alpha })\}}+\mathbf{1}_{\{|-A_{1}+rB_{0}|<\Op(r^{1+\alpha })\}} \right).
	\end{align*}
	Let now $p,q>1,\eta>0 $ such that $(1+\alpha -\eta )/q>1+\alpha '$ . Since also $I_{r}-J_r \neq 0$ implies that either $I_{r}=1$ or $J_r =1$ and so $ | A_{1} | <rD_{r},$ collecting \eqref{eq:666},\eqref{eq:667},
	\begin{align*}
	|a_{r}-&\mathbb{E}^{(r)}\left[|A_{1}^{2}-r^{2}B_{0}^{2}|J_r   \right]  |  \leqslant  | \mathbb{E}^{(r)}( | A_{1}^{2}-r^{2}B_{0}^{2}  |  | I_{r}-J_r  | ) | +\Op(r^{3+\alpha' })\\
	&\leqslant  | \mathbb{E}^{(r)}( | A_{1}^{2}-r^{2}B_{0}^{2} |   \mathbf{1}_{ | A_{1} | <rD_{r}} | I_{r}-J_r  | ) | +\Op(r^{3+\alpha '})\\
	&\leqslant  \mathbb{E}^{(r)}\left[(r^{2}D_{r}^{2}+r^{2}B_{0}^{2})\left(
	\mathbf{1}_{\{|A_{1}+rB_{0}|<\Op(r^{1+\alpha })\}}+\mathbf{1}_{\{|-A_{1}+rB_{0}|<\Op(r^{1+\alpha })\}} 
	\right) 
	\right]+\Op(r^{3+\alpha '})\\&\leqslant  \mathbb{E}^{(r)}[(r^{2}D_{r}^{2}+r^{2}B_{0}^{2})^{p}]^{\frac{ 1}{p}}\left[
\mathbb{P}^{(r)}\left(
	 {|A_{1}+rB_{0}|<\Op(r^{1+\alpha })}
	\right)^{\frac{ 1}{q}}+\mathbb{P}^{(r)}\left(
	 {|-A_{1}+rB_{0}|<\Op(r^{1+\alpha })}
	\right)^{\frac{ 1}{q}}
\right]+\Op(r^{3+\alpha '}).
	\end{align*}
	We have $$\mathbb{P}^{(r)}( | A_{1}+rB_{0} | <\Op(r^{1+\alpha }))\leqslant \mathbb{P}^{(r)}( | A_{1}+rB_{0} | <r^{1+\alpha -\eta })+\mathbb{P}^{(r)}(\Op(1)>r^{-\eta }).$$
	By an application of Lemma \ref{lm:indep-density}   similar to \eqref{eq:11} with $\varphi _{r}(x)$ of the form $\mathbf{1}_{\{ | x_{1}x_{1}+r\sum_{}a_{i,j}x_{i}x_{j} | <r^{1+\alpha -\eta }\}}$
and Lemma \ref{lm:ineg-produits}-(i); the first member is in $r^{1+\alpha -\eta }\ln(r)$, hence finally for some $c<\infty $
\begin{align*}
|a_{r}-&\mathbb{E}^{(r)}\left[|A_{1}^{2}-r^{2}B_{0}^{2}|J_r   \right]|\leqslant cr^{2}r^{(1+\alpha -\eta )/q}\ln(r)^{\frac{ 1}{q}}+\Op(r^{3+\alpha '})=\Op(r^{3+\alpha '}).
\end{align*}
Finally, \eqref{eq:equiv-ar} follows from Lemma \ref{lm:indep-density} with $\varphi _{r}(X')=A_{1}^{2}-r^{2}B_{0}^{2}$.\end{proof}

\subsubsection{Upper bound in \eqref{result od the second moment of extrema points }} 
\label{sec:upper}
According to the previous lemma it suffices to give an upper bound of $\mathbb{E} \left[   \mid  A^{2}_{1} - r^{2} B^{2}_{0}  | J_{r}   \right].$
We stress that the crucial point that justifies the absence of a log term in the final result (compared to \eqref{result od the second moment of saddles points }) is the following  inequality
\begin{align*}
J_r =\mathbf{1}_{\{ | A_{1} | <rB_{0}\}}\leqslant \mathbf{1}_{\{ | \partial_{22}\psi(0) | <2r \partial _{221}\}}+\mathbf{1}_{\{  | \partial _{111}\psi(0) | <\frac{2}{3}r\partial  _{1111}\psi(0)\}},
\end{align*}
hence since $B_{0}^{2}$ is a polynomial in $X'$ we can use Lemma \ref{lm:ineg-produits}-(iii) several times and get for some $c<\infty $
\begin{align}
\label{esperacne A-rb une borne superieure}
\mathbb{E} ( | A_{1}^{2}-r^{2}B_{0}^{2} | J_r  )\leqslant 2\mathbb{E} ( |  r^{2}B_{0}^{2} | J_r  )\leqslant cr^{3}.
\end{align}
%
Then, from \eqref{kEE exrep extremza extrema },\eqref{esperacne A-rb une borne superieure} and \eqref{the 2 point untion crituq critique 3 }, we deduce that for some $c'<\infty $
\begin{equation}
\label{Kee  upper bound}
\mathrm{K}^{e,e}_{2}(z,w) \leqslant c'\; r^3.
\end{equation}
Finally, from \eqref{the facotroial moment kac rice cas general  extrema} and \eqref{Kee  upper bound}, we deduce for some $c''<\infty $
\begin{align*}
\mathbb{E}[ \mathcal{N}^{e}_{\rho} (\mathcal{N}^{e}_{\rho}-1)] \leqslant c''  \;  \rho^7.
\end{align*}\\

\subsubsection{Lower bound in \eqref{result od the second moment of extrema points }}
\label{sec:lower1}
Thanks to Lemma \ref{lm:ar},  it is sufficient to give a lower bound of $ \mathbb{E} (|A_{1}^{2}-r^{2}B_{0}^{2}| \mathbf{1}_{\{|A_{1}|\leqslant rB_{0}\}}|
).$
Let us first assume that the Gaussian field $\psi $ is not a SGRW (Example \ref{ex:grw}), hence the derivatives involved in $X$ and $Y_{0}$ are not linearly linked. Define the event
\begin{align*}
\Omega=&
\{| \partial_{22}\psi(0) | < r,  \frac{ 1}{2}<\partial_{111}\psi(0) <1 ,    | \partial_{211}\psi(0)  | < 1 ,8<   \partial _{122}  , | \partial _{1111}\psi(0) | <1
\}.\\
\end{align*} 
We recall  

\begin{align*}
A_{1}&= \partial_{22}\psi(0) \partial_{111}\psi(0),\\
 B_{0}&= \partial_{221}\;  \partial_{111}\psi(0)-\partial_{211}\psi(0)^2 +  \frac{1}{3} \partial_{22}\psi(0)\; \partial_{1111}\psi(0)\\
 Y_{0}&=(\partial _{1}\psi(0),\partial _{2}\psi(0),\partial _{11}\psi(0),\partial _{12}\psi(0)).
\end{align*}  {Hence under $\Omega $ 
\begin{align*}
 | A_{1} | <&r\\
B_{0}>4-1-\frac{ r }{3}
\end{align*}

Hence for $r$ sufficiently small, $B_{0}>2$, in particular $ | A_{1} | \leqslant r | B_{0} | /2$ and  we obtain 
\begin{align*}
\mathbb{E} (|A_{1}^{2}-r^{2}B_{0}^{2}| \mathbf{1}_{\{ | A_{1} | \leqslant rB_{0}\}}
) &\geqslant \mathbb{E}  \left[ \mathbf{1}_{\Omega }  \mid  r^{2}B_{0}^{2}/4  \mid  \mathbf{1}_{\{B_{0}\geqslant 2A_{1}/r\}} \right]\\
&\geqslant  \; r^{2}\mathbb{P} (\Omega  ).
\end{align*} 
Then since $X$ is non-degenerate, its density is uniformly bounded and the proof is concluded with
\begin{align*}
\mathbb{P}(\Omega )\geqslant cr>0
\end{align*}for some $c>0.$
In the degenerate case of the SGRW, $\partial _{122}\psi(0)=-\partial _{111}\psi(0)$ if $Y_{0}=0$ and we put instead
\begin{align*}
\Omega =&
\{ | \partial _{111}\psi(0) | <r,\frac{ 1}{2}<\partial _{22}\psi(0)<1,\partial _{1111}\psi(0)>19, | \partial _{211} \psi(0)| <1\}\\
\end{align*}
If $Y_{0}=0$ and $\Omega $ is realised, 
\begin{align*}
 | A_{1} | &<r\\
B_{0}=-\partial _{111}\psi(0)^{2}-\partial _{211}\psi(0)^{2}+\frac{ 1}{3}\partial _{22}\partial _{1111}\psi(0)&>-r^{2}-1+\frac{ 19}{6}
\end{align*}
hence $ | A_{1} | <r<B_{0}r/2$ for $r$ small enough and the same method can be applied because $X'$ has a bounded density.
Therefore, it holds for some $c'>0$
\begin{align}
\label{lower de omega1}
\mathbb{E}^{(0)}(|A_{1}^{2}-r^{2}B_{0}^{2}| \mathbf{1}_{\{|A_{1}|\leqslant rB_{0}\}}|)&\geqslant  c ' \; r^3.
\end{align}
From  \eqref{kEE exrep extremza extrema }, \eqref{the 2 point untion crituq critique 3 }and  \eqref{lower de omega1}, we get for some $c''>0$
\begin{align}
\label{Kee LOWER bound}
{\mathrm{K}}^{e,e}_{2}(z,w)&\geqslant   \; c'' r^{3}.
\end{align}
Finally, from \eqref{the facotroial moment kac rice cas general  extrema} and \eqref{Kee LOWER bound}, we deduce that for some $c'''>0$
\begin{align*}
\label{lower bound the moment factorie 2 extrema}
\mathbb{E}[ \mathcal{N}^{e}_{\rho} (\mathcal{N}^{e}_{\rho}-1)] & \geqslant  c''' \rho^{7}.
\end{align*}
\subsection{  Proof of \eqref{result od the second moment of saddles points }  in Theorem \ref{theom second 2 crtique critque and extrema and s}}Using Theorem \ref{thm:KR} with $B_{1}=B_{2}=(-\infty ,0)$, the  second factorial moment of $\mathcal{N}^{s}_{\rho}=N_{\rho }^{(-\infty ,0)}$ is given by 
\begin{equation*}
\label{the facotroial moment kac rice cas general  s}
\mathbb{E}[ \mathcal{N}^{s}_{\rho} (\mathcal{N}^{s}_{\rho}-1)]= \int    \int_{\B_\rho\times   \B_\rho}  {\mathrm{{K}}}^{s,s}_{2}(z,w) \,\mathrm{d}z \; \mathrm{d}w, 
\end{equation*}
where  
\begin{equation*}
\label{2 point correlatio k2 CAS general  saddle }
{\mathrm{K}}^{s,s}_{2}(z,w)= r^2 \;  \phi_{(\nabla \psi (z),\nabla \psi (w)  ) }((0,0)), (0,0)) \;  \;  \mathbb{E}^{(0)} \left[  \; | \det H_{\psi}(z)|\;  | \det H_{\psi}(w)|  \; \mathbf{1}_{\{ \det H_{\psi}(z)<0 \}} \;  \; \mathbf{1}_{\{ \det H_{\psi}(w)<0\}}  \right].
\end{equation*}

The difference is hence on the sign of the determinants, ${\mathrm{K}}^{s,s}_{2}(z,w)$ becomes
\begin{equation}
\label{kEE exrep extremza extrema 2}
{\mathrm{K}}^{s,s}_{2}(z,w)=    r^2  \; \phi_{(\nabla \psi (z),\nabla \psi (w)  )}((0,0),(0,0))  \; a_{r}' \\
\end{equation}
where (see \eqref{detreminant hess au point w})
\begin{align*}
a_{r}':=&\mathbb{E}^{(r)} \left[  \; \big|  {A_{1}}^{2}-g(r)  \big| \;   I_{r}'  \right]\\
I'_{r}:=&\mathbf{1}_{\{A_{1}+r B_{r} +r^2 C_{r}<0\}}  \mathbf{1}_{\{-A_{1}+r B_{r}' +r^2 C'_{r}<0\}}.
\end{align*}
The asymetry of the expression of the determinant yields a different estimate than in the previous case.
To be able  to prove  \eqref{result od the second moment of extrema points }, we need to establish  an upper  bound and a lower bound    of $a_{r}'$ as in the previous section (Lemma \ref{asmpy expression of ar}).  We give in Lemma \ref{asmpy expression of arr} an asymptotic expression of $a_{r}'$.

The proof is similar but there are also significant differences. The difference with respect to before is that the two signs of the determinants are negative, hence we replace $J_{r}$ by 
\begin{align*}
J_{r}'=\mathbf{1}_{\{A_{1}+rB_{0}<0,-A_{1}+rB_{0}<0\}}=  \mathbf{1}_{\{|A_{1}|\leqslant -rB_{0}\}}
\end{align*}
and emphasize that $B_{0}$ does not have the same law as $-B_{0}.$
\begin{lemma}We have for $0<\alpha '<\alpha ,$
	\label{asmpy expression of arr}
	\begin{align*}
	|a_{r}'-\mathbb{E}^{(r)}(|A_{1}^{2}-r^{2}B_{0}^{2}| J_{r}'
	)|=&\Op(r^{3+\alpha '})\\
	a_{r}'\asymp&  \mathbb{E} (|A_{1}^{2}-r^{2}B_{0}^{2}| J'_{r}
	)
	\end{align*}
\end{lemma}

The proof is omitted as the proof of Lemma \ref{lm:ar} can be repdroduced verbatim, with  resp. $J_{r}',I_{r}',a_{r}'$ in place of resp. $J_{r},I_{r},a_{r}.$
\subsubsection{Upper bound}
The upper bound on $J_{r}'$ is of different nature than that on $J_{r}$, in particular the third term
\begin{align*}
J_{r}'=\mathbf{1}_{\{ | A_{1} | <-rB_{0}\}}\leqslant \mathbf{1}_{\{ | \partial_{22}\psi(0) | <-6r \partial _{221}\psi(0) \}}+\mathbf{1}_{\{  | \partial _{111} \psi(0) | <- {2r\partial _{1111}\psi (0)} \}}+\mathbf{1}_{\{ | \partial _{22}\psi(0) \partial _{111} | <3r\partial _{211}\psi(0) ^{2}\}}.
\end{align*}
Then, 
\begin{align*}
\mathbb{E}( | A_{1}-r^{2}B_{0}^{2} | J_{r}')\leqslant \mathbb{E}(2r^{2}B_{0}^{2}J _{r'})
\end{align*} 
hence we must use this time Lemma \ref{lm:ineg-produits}-(ii) for the last term of $J_{r}'$'s bound,
\begin{align*}
\mathbb{E}(B_{0}^{2}\mathbf{1}_{ | \partial _{22}\psi (0)\partial _{111}\psi (0) | <3r\partial _{211}\psi (0)^{2}})\leqslant &\mathbb{E}( | \partial _{22}\psi (0)\partial _{111}\psi (0) |\mathbf{1}_{\{ | \partial _{22}\psi (0)\partial _{111} \psi (0)| <3r\partial _{211}\psi (0)^{2}\}})\\
&+\mathbb{E}( | \partial _{211}\psi (0) | ^{2}\mathbf{1}_{\{ | \partial _{22}\psi (0)\partial _{111}\psi (0) | <3r\partial _{211}\psi (0)^{2}\}}))\\
&+\mathbb{E}( | \partial _{22}\psi (0)\partial _{1111}\psi (0)\mathbf{1}_{\{ | \partial _{22}\psi (0)\partial _{111}\psi (0) | <3r\partial _{211}\psi (0)^{2}\}}))\\
\leqslant& cr\ln(r)
\end{align*}for some $c<\infty .$
The other terms are dealt with by Lemma \ref{lm:ineg-produits}-(iii) as in \eqref{esperacne A-rb une borne superieure},
hence the upper bound is in 
\begin{align*}
\mathbb{E}( | A_{1}-r^{2}B_{0}^{2} | J_{r}')\leqslant c'r^{3}\ln(r)
\end{align*}for some $c'<\infty ,$
which yields \eqref{result od the second moment of saddles points } by \eqref{kEE exrep extremza extrema 2} and
Lemma	\ref{asmpy expression of arr}.

\subsubsection{Lower bound}

We recall the expression of $A_{1}$ and $-B_{0}:$ 
$A_{1}= \partial_{22}\psi(0) \partial_{111}\psi(0), \; -B_{0}= -\partial_{221}\psi(0) \;  \partial_{111}\psi(0)+\partial_{211}\psi(0) ^2 -\frac{1}{3} \partial_{22}\psi(0)\; \partial_{1111}\psi(0) .$ The strategy is the same than at Section \ref{sec:lower1}.\\

If $\psi $ is a SGRW (Example \ref{ex:grw}), $\partial _{111}\psi (0)=-\partial _{122}\psi (0)$ if $Y_{0}=0$, let 
\begin{align*}
\Omega =\{  \partial _{211}\psi (0)>2, | \partial _{22}\psi (0)\partial _{111} \psi (0)| <r, | \partial _{111} \psi (0)| <1, | \partial _{22}\psi (0) | <1, | \partial _{1111} \psi (0)| <1\}.
\end{align*}
Hence if $Y_{0}=0$ and $\Omega $ is realised
\begin{align*}
A_{1}&<r,\\
-B_{0}&=\partial _{111}\psi (0)^{2}+\partial _{211}\psi (0)^{2}-\frac{ 1}{3}\partial _{22}\psi (0)\partial _{1111}\psi (0)>0+4-\frac{ 1}{3}>2A_{1}/r.
\end{align*}
We have 
\begin{align*}
\mathbb{E} (|A_{1}^{2}-r^{2}B_{0}^{2}| \mathbf{1}_{\{ | A_{1} | \leqslant -rB_{0}\}}
) &\geqslant \mathbb{E}  \left[ \mathbf{1}_{\Omega }  \mid  r^{2}B_{0}^{2}/4  \mid  \mathbf{1}_{\{B_{0}\geqslant 2A_{1}/r\}} \right]\\
&\geqslant  \; r^{2}\mathbb{P} (\Omega  ).
\end{align*} 
We must prove a converse to Lemma \ref{lm:ineg-produits}-(i) with $s=0$. Since the  density of $X'$ is uniformly bounded from below on $[-3,3]^{4}$, for some $c>0,$
\begin{align*}
\mathbb{P}(\Omega ) \geqslant c \int_{ [-1,1]^{2}}\mathbf{1}_{\{ | x_{1}x_{2} | <r\}}dx_{1}dx_{2}\asymp r\ln(r). 
\end{align*}

\subsection{Proof of \eqref{preuve crituqe crique vzaleur ac} in Theorem \ref{theom second 2 crtique critque and extrema and s}}
We recall that  $  \mathcal{N}_{\rho}^{c}= \mathcal{N}_{\rho}^{s}+ \mathcal{N}_{\rho}^{e}$ hence $\mathcal{N}_{\rho}^{c}(\mathcal{N}_{\rho}^{c}-1) = \mathcal{N}_{\rho}^{e}(\mathcal{N}_{\rho}^{e}-1) + \mathcal{N}_{\rho}^{s}(\mathcal{N}_{\rho}^{s}-1) + 2 \mathcal{N}_{\rho}^{e}\mathcal{N}_{\rho}^{s}.$\\
So, we have:
\begin{align*}
\mathbb{E}[ \mathcal{N}_{\rho}^{e}\mathcal{N}_{\rho}^{s} ] &=\frac{1}{2}   \mathbb{E}[\mathcal{N}_{\rho}^{c}(\mathcal{N}_{\rho}^{c} -1)]-  \mathbb{E}[\mathcal{N}_{\rho}^{e}(\mathcal{N}_{\rho}^{e} -1)] -  \mathbb{E}[\mathcal{N}_{\rho}^{s}(\mathcal{N}_{\rho}^{s} -1)].
\end{align*}
Combining this formula with previous estimates   \eqref{result of the second factorial moment critical points theroem}, \eqref{result od the second moment of extrema points } and \eqref{result od the second moment of saddles points }, we obtain 
\begin{align*}
\mathbb{E}[ \mathcal{N}_{\rho}^{e}\mathcal{N}_{\rho}^{s}]  &=\frac{1}{2}   \mathbb{E}[\mathcal{N}_{\rho}^{c}(\mathcal{N}_{\rho}^{c} -1)]+o(1)\end{align*}
ending the proof of  \eqref{preuve crituqe crique vzaleur ac}.  

\begin{lemma}
	\label{lm:ineg-produits}
	Let $(Z_{1},\dots ,Z_{k})$ be a non-degenerate Gaussian vector and $a_{i,j}$ real fixed coefficients. Then,\begin{itemize}
\item [(i)] 
\begin{align*}
\mathbb{P}( | Z_{1}Z_{2}+s\sum_{i,j}a_{i,j}Z_{i}Z_{j} | <r)\leqslant Cr\ln(r)\tag{i}
\end{align*}
for $C$ depending on  the law of the $Z_{i}$ (and not on $s$ or the $a_{i,j}$),
\item[ (ii)] for $\alpha _{i}\geqslant 0$ 
\begin{align*}
\tag{ii}
\mathbb{E}( | Z_{1}^{\alpha _{1}}\dots Z_{k}^{\alpha _{k}} | \mathbf{1}_{\{ | Z_{1}Z_{2} | <rZ_{3}^{2}\}})\leqslant C'\begin{cases}r\ln(r)$ if $\alpha _{1}=\alpha _{2}=0\\
r$ otherwise$ \end{cases}
\end{align*}for some $C'<\infty .$
\item[(iii)]Let  some coefficients $\alpha _{i}\in \mathbb{N}$ , $  (Z_{1},\dots ,Z_{q})$ be a Gaussian vector. Then,  for some $C''<\infty ,$
	\begin{align*}
	\tag{iii}  \mathbb{E} (|Z_{1}^{\alpha _{1}}\dots Z_{q}^{\alpha _{q}}|\mathbf{1}_{\{ | Z_{1} | \leqslant  rZ_{2}\}}\  )\leqslant C''r.
	\end{align*}
\end{itemize} 
\end{lemma}
	
	\begin{proof}(i)
	Assume first that the $Z_{i}$ are iid Gaussian. Let us study for $a,b\in \mathbb{R}$, $Y_{1}:=Z_{1}-as,Y_{2}:=Z_{2}-bs$. Since $Y_{1},Y_{2}$ have a density bounded by $
	\kappa <\infty $ (universal), we have for $c\in \mathbb{R}$
\begin{align*}
\mathbb{P}( | Y_{1}Y_{2}-c | \leqslant r)\leqslant &\mathbb{P}( | Y_{2} | \leqslant r)+\mathbb{P}( | Y_{1}-c/Y_{2} | <r/Y_{2}, | Y_{2} | >1)+\mathbb{P}( | Y_{1}-c/Y_{2} |< r/Y_{2}, | Y_{2} | \in [r,1])\\
\leqslant &\kappa r+\mathbb{P}( | Y_{1}-c/Y_{2} | <r)+\mathbb{E}\left[
\mathbb{P}(Y_{1}\in [c/Y_{2}\pm r/Y_{2}]\;|\;Y_{2})\mathbf{1}_{\{r< | Y_{2} | <1\}}
\right]\\
\leqslant &\kappa r+\kappa r+\mathbb{E}(\kappa r/Y_{2}\mathbf{1}_{\{r< | Y_{2} | <1\}})\\
\leqslant &2\kappa r +\kappa r\int_{ r}^{1}\frac{ 1}{y_{2}}2\kappa dy_{2} \\
\leqslant &2\kappa r+2\kappa ^{2}r\ln(r),
\end{align*}
uniformly on $a,b,c,s.$ Then it remains to notice that 
\begin{align*}
Z_{1}Z_{2}+s\sum_{i,j}a_{i,j}Z_{i}Z_{j}=(Z_{1}-As)(Z_{2}-Bs)-C_{s}
\end{align*}
where $A,B,C_{s}$ are independent of $Z_{1},Z_{2}$. Then 
\begin{align*}
\mathbb{P}( | Z_{1}Z_{2}+s\sum_{i,j}a_{i,j}Z_{i}Z_{j} | <r)=\mathbb{E}(\mathbb{P}( | (Z_{1}-As) (Z_{2}-Bs) -C_{s} | \;|\;A,B,C_{s}))\leqslant Cr\ln r.
\end{align*}

In the non-independent Gaussian case, the joint density $f(x_{1},\dots ,x_{k})$ of $(Z_{1},\dots ,Z_{k})$ is bounded by $\kappa \exp(-c\sum_{i}x_{i}^{2})$ for some $c,\kappa >0$ ($c$ would be the smallest eigenvalue of the covariance matrix). From there on the conclusion is easy:
\begin{align*}
\int_{ }\mathbf{1}_{\{ | x_{1}x_{2}+s\sum_{i,j}a_{i,j}x_{i}x_{j} | <r\}}f(x_{1},\dots ,x_{k})dx_{1}\dots dx_{k}\leqslant  \kappa  \int_{ }\mathbf{1}_{\{...\}}\exp(-c\sum_{i}x_{i}^{2})dx_{1}\dots dx_{k}
\end{align*}
	and the right hand member corresponds to the independent case, already treated.
	
	(ii) For the second assertion, assume first that the $Z_{k}$ are independent. Without loss of generality, assume $\alpha _{1}\leqslant \alpha _{2}$. We have  for $t\geqslant 0$, for some $c,c',c'',c'''',C<\infty ,$
\begin{align*}
\mathbb{E}( | Z_{1}^{\alpha _{1}}Z_{2}^{\alpha _{2}} | \mathbf{1}_{ | Z_{1}Z_{2} | }<t)\leqslant &\mathbb{E}( | Z_{1}^{\alpha _{1}}Z_{2}^{\alpha _{2}} | \mathbf{1}_{ | Z_{1} | <t})+\mathbb{E}( | Z_{1}^{\alpha _{1}}Z_{2}^{\alpha _{2}} | \mathbf{1}_{ \{| Z_{2} | <t, | Z_{1} | >1\}})+\mathbb{E}( | Z_{1}^{\alpha _{1}}Z_{2}^{\alpha _{2}} | \mathbf{1}_{\{ | Z_{2} | <t/ | Z_{1} |\} }\mathbf{1}_{\{t< | Z_{1} | <1\}})\\
\leqslant & ct^{\alpha _{1}+1}+ct^{\alpha _{2}+1}+c'\int_{t}^{1}x_{1}^{\alpha _{1}}\int_0^{t/x_{1}}x_{2}^{\alpha _{2}}dx_{2}dx_{1}\\
\leqslant &2ct+c''\int_{t}^{1}x_{1}^{\alpha _{1}}\left(
\frac{ t}{x_{1}}
\right)^{\alpha _{2}+1}dx_{1}\\
\leqslant &2ct+c'''t^{\alpha _{2}+1}\begin{cases}t^{\alpha _{1}-\alpha _{2}}$ if $\alpha _{1}<\alpha _{2}\\
\ln(t)$ if $\alpha _{1}=\alpha _{2} \end{cases}\\
\leqslant &C\begin{cases}t\ln(t)$ if $\alpha _{1}=\alpha _{2}=0\\ t$ otherwise$\end{cases}.
\end{align*}
Coming back to the main estimate with $t=rZ_{3}^{2}$, using conditional expectations, for some $C',C''<\infty ,$
\begin{align*}
\mathbb{E}( | Z_{1}^{\alpha _{1}}\dots Z_{k}^{\alpha _{k}} | \mathbf{1}_{\{ | Z_{1}Z_{2} | <rZ_{3}^{2}\}})\leqslant C' \mathbb{E}(\prod_{i\neq 1,2,3}Z_{k}^{\alpha _{k}}(rZ_{3}^{\alpha _{3}+2}\ln(rZ_{3})^{\mathbf{1}_{\alpha _{1}=\alpha _{2}=0}}))\leqslant C'' r\ln(r)^{\mathbf{1}_{\alpha _{1}=\alpha _{2}=0}}.
\end{align*}
The non-independent (non-degenerate) case can be treated as before by bounding the density of the $Z_{k}$ by an independent density of the same order.

 (iii)
	 By Holder's inequality 
\begin{align*}
\mathbb{E} (|Z_{1}^{\alpha _{1}}\dots Z_{q}^{\alpha _{q}}|\mathbf{1}_{\{ | Z_{1} | \leqslant crZ_{2}\}}\  )\leqslant \prod_{i=1}^{q}\mathbb{E}( | Z_{i} | ^{q\alpha _{i}}\mathbf{1}_{\{ | Z_{1} | \leqslant crZ_{2}\}})^{\frac{ 1}{q}}
\end{align*}
hence we can assume wlog that only one $\alpha _{i}$, say $\alpha _{i_{0}}$, is non-zero. 
For $i_{0}>2,$ we have an orthogonal  decomposition of the form $Z_{i_{0}}=(\alpha  Z_{1}+\beta  Z_{2})+\gamma  Y $ where $Y $ is independent of $(Z_{1},Z_{2})$, hence we can assume wlog that $i_{0}=1$ or $i_{0}=2$. For $i_{0}=1$, the bound is 
\begin{align*}
\mathbb{E}( | rZ_{2} | ^{\alpha _{1}}\mathbf{1}_{\{ | Z_{1} | <rZ_{2}\}})=O(r^{1+\alpha _{1}})
\end{align*} 
and it only remains to treat the case $i_{0}=2$. In this case we decompose orthogonally $Z_{1}=\lambda Z_{2}+\mu Z$ where $Z$ is independent of $Z_{2}$. Then the bounded densities of $Z_{2}$ and $Z$ easily yields the result
\begin{align*}
\mathbb{E}( | Z_{2} | ^{\alpha _{2}} C| rZ_{2} | )=O(r).
\end{align*}

\begin{minipage}{12cm}
\section*{Acknowledgements}
We warmfully thank Anne Estrade, who participated to the conception of the project,   the supervision of this work and to the elaboration of this article. This project has received funding from the European Union's Horizon 2020 research and innovation programme under the Marie Sklodowska-Curie  grant agreement No 754362.\\
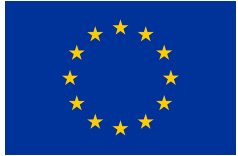
\begin{center} 
\includegraphics[scale=.15]{eu_flag.jpeg} 
\end{center}
\end{minipage}

	\end{proof}

\bibliographystyle{alpha}


\begin{thebibliography}{WMNE96}

\bibitem[AD22]{azais2019mean}
J.~Azais and C.~Delmas.
\newblock Mean number and correlation function of critical points of isotropic
  {G}aussian fields.
\newblock {\em Stoc. Proc. Appl.}, 150:411--445, 2022.

\bibitem[AT09]{adler2009random}
R.~J. Adler and J.~E. Taylor.
\newblock {\em Random fields and geometry}.
\newblock Springer Science \& Business Media, 2009.

\bibitem[ATW07]{adler2007applications}
R.~J. Adler, J.~E. Taylor, and K.~Worsley.
\newblock Applications of random fields and geometry: Foundations and case
  studies.
\newblock In {\em In preparation, available on R. Adler home page}. Citeseer,
  2007.

\bibitem[AW09]{azais2009level}
J.~Aza{\"\i}s and M.~Wschebor.
\newblock {\em Level sets and extrema of random processes and fields}.
\newblock John Wiley \& Sons, 2009.

\bibitem[BCW19]{beliaev2019two}
D.~Beliaev, V.~Cammarota, and I.~Wigman.
\newblock Two point function for critical points of a random plane wave.
\newblock {\em International Mathematics Research Notices}, 2019(9):2661--2689,
  2019.

\bibitem[BH20]{BarHar}
R.~Bardenet and A.~Hardy.
\newblock Monte carlo with determinantal point processes.
\newblock {\em Ann. Appl. Prob.}, 30(1):368--417, 2020.

\bibitem[CG13]{chevalier2013fast}
C.~Chevalier and D.~Ginsbourger.
\newblock Fast computation of the multi-points expected improvement with
  applications in batch selection.
\newblock In {\em International Conference on Learning and Intelligent
  Optimization}, pages 59--69. Springer, 2013.

\bibitem[CS17]{cheng2017multiple}
D.~Cheng and A.~Schwartzman.
\newblock Multiple testing of local maxima for detection of peaks in random
  fields.
\newblock {\em Annals of statistics}, 45(2):529, 2017.

\bibitem[CX16]{cheng2016mean}
D.~Cheng and Y.~Xiao.
\newblock The mean {E}uler characteristic and excursion probability of
  {G}aussian random fields with stationary increments.
\newblock {\em The Annals of Applied Probability}, 26(2):722--759, 2016.

\bibitem[DGL17]{desolneux:hal-01548767}
A.~Desolneux, B.~Galerne, and C.~Launay.
\newblock {Etude de la r{\'e}pulsion des processus pixelliques
  d{\'e}terminantaux}.
\newblock Juan les pins, France, September 2017.

\bibitem[Fyo04]{fyodorov2004complexity}
Y.~V. Fyodorov.
\newblock Complexity of random energy landscapes, glass transition, and
  absolute value of the spectral determinant of random matrices.
\newblock {\em Physical review letters}, 92(24):240601, 2004.

\bibitem[KT12]{KulTas}
A.~Kulesza and B.~Taskar.
\newblock Determinantal point processes for machine learning.
\newblock {\em Found. Trends Mach. learn.}, 5(2-3), 2012.

\bibitem[LGD21]{LGA}
C.~Launay, B.~Galerne, and A.~Desolneux.
\newblock Determinantal point processes for image processing.
\newblock {\em SIAM Journal on Imaging Sciences}, 14(1), 2021.

\bibitem[Lin72]{lindgren1972local}
G.~Lindgren.
\newblock Local maxima of {G}aussian fields.
\newblock {\em Arkiv f{\"o}r matematik}, 10(1-2):195--218, 1972.

\bibitem[LW04]{larson2004hot}
D.~L. Larson and B.~D. Wandelt.
\newblock The hot and cold spots in the wilkinson microwave anisotropy probe
  data are not hot and cold enough.
\newblock {\em The Astrophysical Journal Letters}, 613(2):L85, 2004.

\bibitem[MRV]{MRV}
S.~Muirhead, A.~Rivera, and H.~Vanneuville.
\newblock The phase transition for planar {G}aussian percolation models without
  fkg.
\newblock https://arxiv.org/abs/2010.11770.

\bibitem[Mui20]{muirhead2020second}
S.~Muirhead.
\newblock A second moment bound for critical points of planar {G}aussian fields
  in shrinking height windows.
\newblock {\em Statistics \& Probability Letters}, 160:108698, 2020.

\bibitem[NH03]{nichols2003controlling}
T.~Nichols and S.~Hayasaka.
\newblock Controlling the familywise error rate in functional neuroimaging: a
  comparative review.
\newblock {\em Statistical methods in medical research}, 12(5):419--446, 2003.

\bibitem[NPR19]{NPR}
I.~Nourdin, G.~Peccati, and M.~Rossi.
\newblock Nodal statistics of planar random waves.
\newblock {\em Comm. Math. Phys.}, 369:99--151, 2019.

\bibitem[Pot09]{Potthof}
J.~Potthof.
\newblock Sample properties of random fields. ii. continuity.
\newblock {\em Comm. Stoc. Anal.}, 3(3):331--348, 2009.

\bibitem[TBA19]{TBA}
N.~Tremblay, S.~Barthelm{{\'e}}, and P.~Amblard.
\newblock Determinantal point processes for coresets.
\newblock {\em J. Mach. Learn. Res.}, 20:1--70, 2019.

\bibitem[TW07]{taylor2007detecting}
J.~E. Taylor and K.~J. Worsley.
\newblock Detecting sparse signals in random fields, with an application to
  brain mapping.
\newblock {\em Journal of the American Statistical Association},
  102(479):913--928, 2007.

\bibitem[WMNE96]{worsley1996searching}
K.~J. Worsley, S.~Marrett, P.~Neelin, and A.C. Evans.
\newblock Searching scale space for activation in pet images.
\newblock {\em Human brain mapping}, 4(1):74--90, 1996.

\bibitem[WTTL04]{worsley2004unified}
K.~J. Worsley, J.~E. Taylor, T.~F. Tomaiuolo, and J.~Lerch.
\newblock Unified univariate and multivariate random field theory.
\newblock {\em Neuroimage}, 23:S189--S195, 2004.

\end{thebibliography}
\end{document}